\documentclass[a4paper,reqno]{paper}
\usepackage[utf8]{inputenc}
\usepackage[T1]{fontenc}
\usepackage{amsfonts,enumerate}
\usepackage{amsmath}
\usepackage{amssymb}
\usepackage{amsthm}
\usepackage{graphics}
\usepackage{hyperref}
\usepackage[capitalize]{cleveref}
\usepackage{tikz}
\usepackage{geometry}
\usepackage{amsfonts}
\usepackage{graphicx}%
\usepackage{genyoungtabtikz}
\usepackage{dsfont}
\usepackage{mathrsfs}
\usepackage{ytableau}

\newtheorem{Th}{Theorem}[section]
\newtheorem{Lem}[Th]{Lemma}
\newtheorem{Cor}[Th]{Corollary}
\newtheorem{Prop}[Th]{Proposition}
\newtheorem{Def-Prop}[Th]{Definition-Proposition}

\newtheorem{problem}[Th]{Problem}
\theoremstyle{definition}
\newtheorem{Def}[Th]{Definition}
\newtheorem{Exa}[Th]{Example}

\newtheorem{Rem}[Th]{Remark}
\newcommand{\Z}{\mathbb{Z}}

\newcommand{\C}{\mathbb{C}}

\newcommand{\al}{\alpha}
\newcommand{\la}{\lambda}
\newcommand{\om}{\omega}
\newcommand{\eps}{\varepsilon}

\newcommand{\qu}[1]{\quad{\text{#1}}\quad}

\newcommand{\wh}[1]{\widehat{#1}}
\newcommand{\scal}[2]{\langle#1,#2\rangle}

\newcommand{\se}{\mathsf{e}}
\newcommand{\sv}{\mathsf{v}}
\newcommand{\sx}{\mathsf{x}}
\newcommand{\sE}{\mathsf{E}}
\newcommand{\sI}{\mathsf{I}}

\newcommand{\sV}{\mathsf{V}}
\newcommand{\fs}{\mathsf{s}}

\newcommand{\wP}{\widetilde{\mathcal{P}}}
\newcommand{\hmu}{\widehat{\mu}}
\newcommand{\hla}{\widehat{\la}}

\newcommand{\cC}{\mathcal{C}}
\newcommand{\cP}{\mathcal{P}}

\newcommand{\bX}{\boldsymbol{X}}
\newcommand{\ba}{\boldsymbol{a}}

\newcommand{\bmu}{\boldsymbol{\mu}}
\newcommand{\bnu}{\boldsymbol{\nu}}

\newcommand{\bm}{\boldsymbol{m}}
\newcommand{\bk}{\boldsymbol{k}}

\newcommand{\tc}{{\tilde{c}}}
\newcommand{\tb}{{\tilde{b}}}

\newcommand{\w}{\mathrm{w}}
\newcommand{\hw}{\mathrm{hw}}
\newcommand{\sh}{\mathrm{sh}}
\newcommand{\wt}{\text{\rm wt}}

\newcommand{\fg}{\mathfrak{g}}
\DeclareMathOperator{\Char}{char}
\newcommand{\lra}{\longrightarrow}
\newcommand{\indic}{\mathds{1}}
\allowdisplaybreaks

\definecolor{asse}{RGB}{3,192,94}

\newcommand{\de}{\dot{e}}
\newcommand{\df}{\dot{f}}
\newcommand{\mJ}{\mathrm{J}}

\newcommand*\samethanks[1][\value{footnote}]{\footnotemark[#1]}

\let\svthefootnote\thefootnote
\newcommand\freefootnote[1]{%
  \let\thefootnote\relax%
  \footnotetext{#1}%
  \let\thefootnote\svthefootnote%
}

\begin{document}

\sloppy

\title{Generalised Howe duality
and injectivity of induction: the symplectic case}
\author{
Thomas Gerber\thanks{\'Ecole Polytechnique F\'ed\'erale de Lausanne, Route cantonale, 1015 Lausanne, Switzerland. Supported by an \textit{Ambizione} fellowship of the Swiss National Science Foundation.
Email address: {\tt thomas.gerber@epfl.ch}
},
J\'er\'emie Guilhot\thanks{Institut Denis Poisson, Universit\'e de Tours, Parc de Grandmont, 37200 Tours, France.
Email addresses: {\tt jeremie.guilhot@lmpt.univ-tours.fr}, {\tt cedric.lecouvey@lmpt.univ-tours.fr}}
and
C\'edric Lecouvey\samethanks
}
\maketitle
\freefootnote{}

%%%%%%%%%%%%%%%%%%%%%%%%%%%%%%%%%%%%%%%%%%%%%%%%%%
%%%%%%%%%%%%%%%%%%%%%%%%%%%%%%%%%%%%%%%%%%%%%%%%%%
%%%%%%%%%%%%%%%%%%%%%%%%%%%%%%%%%%%%%%%%%%%%%%%%%%
%%%%%%%%%%%%%%%%%%%%%%%%%%%%%%%%%%%%%%%%%%%%%%%%%%

\hrule 

\begin{abstract}

We study the symplectic Howe duality using two new and independent combinatorial methods:
via determinantal formulae on the one hand, and via (bi)crystals on the other hand.
The first approach allows us to establish
a generalised version where weight multiplicities are replaced by
branching coefficients.
In turn, this generalised Howe duality is used to prove the
injectivity of induction for Levi branchings as previously
conjectured by the last two authors.
\end{abstract}

\hrule

\section{Introduction}

Let $n$ and $m$ be two positive integers and let $\mu$ be a partition whose
Young diagram is contained in the rectangle~$n\times m$ with $n$ rows and $m$
columns. Then the conjugate partition~$\mu ^{\prime }$ of $\mu $ can be
written~$\mu ^{\prime }=(\mu _{1}^{\prime },\ldots ,\mu _{m}^{\prime })$.\ A
classical result in the representation theory of linear Lie algebras
states that the multiplicitiy of the irreducible $\mathfrak{gl}_{n}(\mathbb{C})$-module
 $V(\la)$ in the tensor product 
\begin{equation*}
{\sf \Lambda} _{n,m}^{\mu ^{\prime }}={\sf \Lambda}^{\mu'_1}(\C^{n})\otimes \ldots \otimes {\sf \Lambda}^{\mu'_m}(\C^{n})
\end{equation*}%
is equal to the dimension of the $\mu'$-weight space
in the irreducible $\mathfrak{gl}_{m}(\mathbb{C})$-modules $V(\la')$.\ 
There exist
numerous proofs of this duality in type $A$ (sometimes referred as the Schur
duality in the literature) mostly based on computations on Schur functions
or on purely combinatorial arguments using semistandard tableaux~\cite
{FH91}.

\medskip

For the other classical Lie algebras (or classical Lie groups), there exist
similar constructions due to Howe~\cite{H95}. In this paper,
we restrict ourselves to the symplectic case (i.e. to the root systems of type 
$C$).\ Then, the Howe duality states that the multiplicity of the irreducible $%
\mathfrak{sp}_{2n}(\mathbb{C})$-module $V(\la)$ in the tensor product 
\begin{equation}
{\sf \Lambda} _{2n,m}^{\mu ^{\prime }}={\sf \Lambda} ^{\mu _{1}^{\prime }}(\mathbb{C}^{2n})
\otimes \cdots \otimes {\sf \Lambda} ^{\mu _{m}^{\prime }}(\mathbb{C}^{2n})
\label{TensorHowe}
\end{equation}%
is equal to the dimension of the $\wh{\mu}$-weight space
in the irreducible $\mathfrak{gl}_{m}(\mathbb{C})$-module $V(\wh{\la})$. 
In contrast with the type $A$ case, the conjugate partitions $\mu'$ and $\la'$ are here replaced by the partitions $\wh{\mu}$ and $\wh{\la}$ defined as the conjugates of the 
complements of $\mu$ and $\la$ in the rectangle $n\times m$.

\medskip

The goal of this paper is three-fold. Firstly, we give a simple combinatorial
proof of the Howe duality in the symplectic case based on the determinantal formulae for Weyl characters\ (analogue to the Jacobi-Trudi formulae for
the Schur polynomials). The tools and computations that we use in our proof
generalise those developed in \cite{Le06} and extend naturally to the case where the
fundamental $\mathfrak{gl}_{2n}(\mathbb{C})$-modules appearing in the tensor
products~(\ref{TensorHowe}) are replaced by tensor products of simple $%
\mathfrak{gl}_{2n}(\mathbb{C})$-modules (restricted to $\mathfrak{sp}_{2n}(\mathbb{C})$) or 
simple~$\mathfrak{sp}_{2n}(\mathbb{C})$-
modules.\ We prove that the corresponding tensor product multiplicities are
equal this time to branching coefficients corresponding to the
restriction of the simple $\mathfrak{sp}_{2m}(\mathbb{C})$-modules to a block
diagonal subalgebras $\mathfrak{s}$ of~$\mathfrak{sp}_{2m}(\mathbb{C})$.
This means that we need to consider restrictions to subalgebras $\mathfrak{s}%
\simeq \mathfrak{g}_{1}\oplus \cdots \oplus \mathfrak{g}_{r}$ where each $%
\mathfrak{g}_{i}$ is a Lie algebra isomorphic to $\mathfrak{sp}_{2m_{i}}$ or 
$\mathfrak{gl}_{m_{i}}$ with $m_{1}+\cdots +m_{r}=m$. This generalises the original Howe
duality which corresponds to the case~$r=m$ and  $m_{i} = 1 $ for all $i$.

\medskip

Secondly, the previous Schur and Howe dualities can be generalised when $\mu'$
is replaced by any $m$-tuple $\beta $ of nonnegative integers.\ Then the
spaces
\begin{equation*}
{\sf \Lambda}_{n,m}^{A}=\bigoplus_{\substack{\beta=(\beta_{1},\ldots,\beta_{m}
		)\in\mathbb{Z}_{\geq0}^{m}\\\text{\textrm{max}}(\beta)\leq n}}
		{\sf \Lambda}_{n,m}^{\beta}\text{\quad and \quad}
{\sf \Lambda}_{2n,m}^{C}=\bigoplus
_{\substack{\beta=(\beta_{1},\ldots,\beta_{m})\in\mathbb{Z}_{\geq0}
		^{m}\\\text{\textrm{max}}(\beta)\leq2n}}{\sf \Lambda}_{2n,m}^{\beta}
\end{equation*}
admit a structure of $\mathfrak{gl}_{n}\times \mathfrak{gl}_{m}$-bimodule and of 
$\mathfrak{sp}_{2n}\times \mathfrak{sp}_{2m}$-bimodule respectively. It is well known
 that a lot of information about simple modules associated
to a simple Lie algebras is encoded by particular combinatorial structures studied by Lusztig, Kashiwara and Littelmann: their crystal graph. 
It then makes sense to look for bicrystal structures associated to ${\sf \Lambda}
_{n,m}^{A}$ and ${\sf \Lambda} _{2n,m}^{C}$. The second main result of the paper uses
the combinatorial duality techniques developed in \cite{GL20} to get a
simple bijection between the highest weight vertices in the $\mathfrak{sp}_{2n}$-crystal
$B({\sf \Lambda} _{2n,m}^{C})$ 
associated to ${\sf \Lambda} _{2n,m}^{C}$ and
the King tableaux (a particular model of tableaux counting the weight
multiplicities in type $C_{m},$ see \cite{K76}).\ 
This is reminiscent of a result by Lee \cite{Lee19}
expressed in a different combinatorial language.\ 
In contrast to \cite{Lee19}, 
where a type $\mathfrak{sp}_{2n}\times \mathfrak{sp}_{2m}$-bicrystal
structure is proposed, we then study the action of the type $A_{2m-1}$-crystal operators on $B({\sf \Lambda}_{2n,m}^{C})$.\ 
These are indeed the operators which are in connection with the charge statistics defined in \cite%
{LL18}. In particular, the actions of the $A_{2m-1}$-crystal operators
associated to nodes with unbarred label correspond to contraction operations
on columns of type $C_{n}$ in ${\sf \Lambda}_{2n,m}^{C}$ whereas the action
corresponding to nodes with barred label yield jeu de taquin operations on the
positive or negative part of these columns. This leads to an intriguing statistics on tensor products of type $C_n$ columns 
which does not coincide with the intrinsic energy defined from their affine crystal structure.

\medskip

Finally, our third objective is to use the generalised Howe duality to prove
a conjecture by the last two authors. Consider  a Levi
subalgebra $\mathfrak{l}$ of $\mathfrak{sp}_{2m}(\mathbb{C})$ and $\nu ^{(1)},\nu ^{(2)}$
two dominant weights for $\mathfrak{l}$. The conjecture claims that the two $%
\mathfrak{sp}_{2m}(\mathbb{C})$-modules obtained by induction from $\nu
^{(1)}$ and $\nu ^{(2)}$ are isomorphic if and only if $\nu ^{(1)}$ and $\nu
^{(2)}$ coincide up to an automorphism of the Dynkin diagram associated to $%
\mathfrak{l}$ (or equivalently up to permutation of the components in $\nu
^{(1)}$ and $\nu ^{(2)}$ associated to isomorphic simple subalgebras). It was 
proved in \cite{GL} under restrictive conditions on $\nu ^{(1)}$ and $\nu^{(2)}$. Here we
prove this conjecture in full generality and obtain in fact a more general result in which $\mathfrak{l}$ 
can also be replaced by any direct sum $\mathfrak{s}$ of subalgebras of type $C$.\
The proof uses our generalised Howe duality and some elegant results by Rajan \cite%
{Ra14} on the irreducibility of Weyl characters.

\medskip

Most of the techniques and results developed in this paper can be extended
to the orthogonal types.\ There are nevertheless complications due to the
existence of the spin representations and the lack of a natural analogue of
King tableaux in the duality context which is relevant for the paper. This
will be addressed in \cite{GGL}. We also tried to make the paper more accessible by starting with proofs of some known results using methods 
that will be central for the generalisations that we propose here and in \cite{GGL}.

\medskip 

The paper is organised as follows. In Section 2, we review some well-known results on the combinatorics of root systems and Lie algebras, mainly to set up the notations that we will use. In Section 3, we use the Jacobi-Trudi
formula for Schur functions to (re)prove the Schur duality. This allows us to
introduce the main tools and methods which will be reinvested in Section 4
where Howe duality is derived similarly from determinantal identities for
the Weyl characters of type $C_{n}$. In Section 5, we use crystals to define a natural
combinatorial duality between highest weight vertices in $B({\sf \Lambda} _{2n,m}^{C})$ and King tableaux, thereby giving a bijective proof of the Howe duality. 
We then study in Section 6 the behavior of $B({\sf \Lambda} _{2n,m}^{C})$ under some crystal operators of type $A_{2m-1}$. 
Section 7 is devoted to establishing the generalised Howe duality using determinantal techniques similar to that of Section 4.
Finally, we prove the generalised version of the conjecture of \cite{GL} in Section 8.

%%%%%%%%%%%%%%%%%%%%%%%%%%%%%%%%%%%%%%%%%%%%%%%%%%
%%%%%%%%%%%%%%%%%%%%%%%%%%%%%%%%%%%%%%%%%%%%%%%%%%
%%%%%%%%%%%%%%%%%%%%%%%%%%%%%%%%%%%%%%%%%%%%%%%%%%
%%%%%%%%%%%%%%%%%%%%%%%%%%%%%%%%%%%%%%%%%%%%%%%%%%
%%%%%%%%%%%%%%%%%%%%%%%%%%%%%%%%%%%%%%%%%%%%%%%%%%
%%%%%%%%%%%%%%%%%%%%%%%%%%%%%%%%%%%%%%%%%%%%%%%%%%
%%%%%%%%%%%%%%%%%%%%%%%%%%%%%%%%%%%%%%%%%%%%%%%%%%

\section{Generalities and settings}
\label{settings}
Let $\frak{g}$ be a Lie algebra with root system $R$ with triangular decomposition 
\begin{equation*}
\mathfrak{g=}\bigoplus\limits_{\alpha \in R^+}\mathfrak{g}_{\alpha }\oplus 
\mathfrak{h}\oplus \bigoplus\limits_{\alpha \in R^+}\mathfrak{g}_{-\alpha }
\end{equation*}%
so that $\mathfrak{h}$ is the Cartan subalgebra of $\mathfrak{g}$ and $R^{+}$
its set of positive roots. 

\medskip

We assume that $R$ is realised in a real Euclidean space $V$ with inner product $\langle\cdot,\cdot \rangle$. Let $\Delta$ be the simple system associated to $R$.
For $\al\in R$, let $\al^\vee = \dfrac{2\al}{\langle \al,\al\rangle}$ be the coroot associated to $\al$.  The Weyl group $W$ of $R$ is the group generated by the orthogonal reflections $s_{\al}$ with respect to the hyperplanes 
$$H_{\al} := \{x\in V\mid \scal{x}{\al^\vee} = 0\}.$$
The group $W$ is a finite Coxeter group with distinguished set of generators $S = \{s_\al\mid \al\in \Delta\}$. 
The closure of the connected components of the set $V\setminus \cup_{\al\in R} H_\al$ are called the Weyl chambers. The fundamental Weyl chamber is defined by 
$${\cal C}_0 = \{x\in V\mid \scal{x}{\al}> 0 \text{ for all $\al\in R^+$}\}.$$
For all $\al\in \Delta$, we define the fundamental weight $\om_\al\in V$ by the relations 
$$\scal{\om_\al}{\beta^\vee} = \delta_{\al,\beta}\quad\text{for all $\beta\in \Delta$}. $$
Let $P =\oplus_{\al\in \Delta} \Z \om_{\al}$ be the weight lattice and $P^+$ be the set of dominant weights:
$$P^+ =\{\la\in P\mid \scal{\la}{\al^\vee}\geq 0 \text{ for all $\al\in \Delta$}\} = P\cap {\cal C}_0.$$

\begin{Exa}[Root system of type $A$] 
\label{type_A}
Let $E$ be the Euclidean space of dimension $m$ with basis $(\eps_1,\ldots,\eps_m)$ and let $V$ be the quotient  of $E$ by the subspace generated by $\sum_{i=1}^m\eps_i$. The space $V$ is of dimension $m-1$ and can be endowed with an euclidean structure as it is isomorphic to the hyperplane of $E$ orthogonal to $\eps_1+\cdots+\eps_m$.
The set 
$$R^+ = \{\eps_{i}-\eps_j\mid i,j\in \{1,\ldots,m\}, i<j\}$$
is a positive root system of type $A_{m-1}$. For all $i\in \{1,\ldots,m-1\}$ we set $\al_i = \eps_i- \eps_{i+1}$ and we denote by $s_i$ the reflection with respect to $\al_i$. The simple system associated to $R_+$ is 
$$\Delta = \{\al_i\mid i=1,\ldots,m-1\}.$$
The reflection $s_{\eps_i-\eps_j}$ acts on $V$ by permuting the $i$-th and $j$-th coordinate. We have
$${\cal C}_0 = \{(x_1,\ldots,x_m)\in V\mid x_i - x_{i+1}\geq 0 \text{ for all $i$}\}.$$
For all $i\in \{1,\ldots,m-1\}$, the fundamental weight associated to $\al_i$ is $\om_i = \eps_1+\ldots+\eps_i$ and 
$$P^+ = \{(\la_1,\ldots,\la_m)\in \Z^m\mid \la_i - \la_{i+1}\geq 0 \text{ for all $i$}\}.$$
The set $P^+$ is in bijection with the set of partition of length at most $m-1$. Indeed 
$$(\la_1,\ldots,\la_m) = (\la_{1}-\la_m,\ldots,\la_{m-1} -\la_m,0) \text{ in $V$}.$$

\end{Exa}

\begin{Exa}[Root system of type $C$] 
\label{type_C}Let $V$ be the Euclidean space of dimension $m$ with basis $(\eps_1,\ldots,\eps_m)$.
The set 
$$R^+ = \{\eps_{i}-\eps_j\mid i,j\in \{1,\ldots,m\}, i<j\} \cup \{2\eps_i\mid i =1,\ldots,m\}$$
is a positive root system of type $C_{m}$. The associated simple system is 
$$\Delta = \{\al_i\mid i=1,\ldots,m-1\}\cup  \{2\eps_m\}\qu{where} \al_i = \eps_i - \eps_{i+1}.$$
We denote by $s_i$ the reflection $s_{\al_i}$ and by $s_m$ the reflection $s_{2\eps_m}$. 
Note that $s_{\eps_i-\eps_j}$ acts on $V$ by permuting the $i$-th and $j$-th coordinate and the reflection $s_{2\eps_i}$ changes the sign of the $i$-th coordinate. 
We have
$${\cal C}_0 = \{(x_1,\ldots,x_m)\in V\mid x_i - x_{i+1}\geq 0 \text{ for all $i$ and } x_m\geq 0\}.$$
and 
$$P^+ = \{(\la_1,\ldots,\la_m)\in \Z^m\mid \la_i - \la_{i+1}\geq 0 \text{ for all $i$ and }\la_m \geq 0\}.$$
The set $P^+$ is in bijection with the set of partitions of length at most $m$. 
\end{Exa}

\medskip

We now turn to the representation theory of $\mathfrak{g}$.
The \textit{representation ring} of $\fg$ is the ring with basis indexed by the isomorphism classes $[V]$ of irreducible representations $V$ of $\fg$ over $\C$. 
The addition is defined such that $[V] + [V'] =[V'']$ whenever $V''\simeq V\oplus V'$ and the multiplication is defined by $[V]\times [W]= [V\otimes W]$. 
We will denote it by $\mathcal{R}(\fg)$.
\medskip

Let $\Z[P]$ be the integral group ring on the abelian group $P$. We will write $e^{\la}$ for the element associated to $\la\in P$ so that we have $e^{\la}\cdot e^{\la'} = e^{\la+\la'}$. The Weyl group $W$ of $R$ acts naturally on $\Z[P]$ by setting $w\cdot e^{\la} = e^{w(\la)}$.  We denote by $\Z[P]^W$ the set of fix points:
$$\Z[P]^{W} = \{f\in \Z[P] \mid w\cdot f = f\}.$$ Let $\Char$ be the injective ring homomorphism  from the representation ring of $\fg$ to $\Z[P]$  defined by 
$$\Char([V]) = \sum \dim(V_\mu) e^{\mu}$$
where $V_\mu$ is the $\mu$-weight space in $V$. 
For $\la\in P$, let $V(\la)$ be the irreducible module of highest weight $\la$ in~$\frak{g}$. 
\begin{Rem}
\label{notation}
In this paper, we will deal with modules for Lie algebras of type $A$ or $C$ of various rank. When necessary, we will add a superscript to the notation $V(\la)$ to indicate the type and the rank of the Lie algebra we are working with. For instance, we will write $V^{C}(\la)$ for the irreducible module of highest weight $\la$ in a Lie algebra of type $C$ and $V^{C_m}(\la)$ if we further want to indicate the rank $m$ of the Lie algebra. 
\end{Rem}
\begin{Th}[{\cite[Theorem 23.24]{FH91}}]
The ring $\mathcal{R}(\fg)$ is a polynomial ring in the variables $(\Char([V(\om_\al)]))_{\al\in \Delta}$. The homomorphism $\Char:\mathcal{R}(\fg)\rightarrow \Z[P]^W$ is an isomorphism. 
\end{Th}
Let $\rho$ be the half sum of positive roots. Then the Weyl character formula asserts that, for $\la\in P^+$, the character of $V(\la)$ is
$$\fs_{\la} = \dfrac{a_{\la+\rho}}{a_\rho}\qu{where} a_{\la} = \sum_{w\in W}\eps(w)e^{w\la}.$$
The set $\{\fs_{\la}\mid \la\in P^+\}$ is a basis of $\Z[P]^W$. The Kostant partition function ${\cal P}$ is defined by the formula
$$\dfrac{1}{\prod_{\al\in R^+}(1-e^\al)} = \sum_{\beta\in \Z^m} \mathcal{P}(\beta)e^\beta.$$
The dot action of the Weyl group on $P$ is defined by 
$$w\circ \la = w(\la+\rho) - \rho = t_{-\rho}wt_{\rho} (\la),$$
where, for all $\gamma\in P$, $t_\gamma$ denotes the translation by $\gamma$.
\begin{Th}[Kostant mutliplicity formula]\label{Kostant}
Let $\la,\mu\in P^+$. Let $K_{\la,\mu}$ be the dimension of the $\mu$-weight space in the irreducible representation $V(\la)$ of highest weight $\la$. We have 
$$K_{\la,\mu} = \sum_{w\in W} \eps(w)\cP(w\circ \la - \mu).$$
\end{Th}

\begin{Exa} 
Assume that $\fg=\mathfrak{sl}_{n}(\C)$.  Then the root system  $R$ of $\fg$  is of type $A_{n-1}$ as in Example \ref{type_A}, and we have $W=\mathfrak{S}_n$. 
We set $x_i = e^{\eps_i}\in \Z[P]$. Since we realised $R$ in the quotient space $V$ we have the relation $x_1\cdots x_n = 1$ in $\Z[P]$. As a consequence, one can show that
$$\Z[P]^W\simeq \Z[x_1,\ldots,x_n]^{\mathfrak{S}_n}\slash \langle x_1\cdots x_n - 1\rangle$$
where $\Z[x_1,\ldots,x_n]^{\mathfrak{S}_n}$ is the ring of symmetric polynomials in $n$ variables. The representations $V(\om_k)$ for~$k=1,\ldots,n-1$ are isomorphic to ${\sf \Lambda}^{k}(\C^{n})$, the $k$-th exterior power of the natural representation $\C^n$ of $\fg$. Then we have $\Char(V(\om_k)) = \se_k(x_1,\ldots,x_n)\in \Z[P]^W$ where $\se_k$ is $k$-th elementary symmetric function in $n$ variables:
$$\se_i(x_1,\ldots,x_n) = \sum_{1\leq i_1<i_2<\ldots<i_k\leq n} x_{i_1}\cdots x_{i_k}.$$
\end{Exa}

%%%%%%%%%%%%%%%%%%%%%%%%%%%%%%%%%%%%%%%%%%%%%%%%%%
%%%%%%%%%%%%%%%%%%%%%%%%%%%%%%%%%%%%%%%%%%%%%%%%%%
%%%%%%%%%%%%%%%%%%%%%%%%%%%%%%%%%%%%%%%%%%%%%%%%%%
%%%%%%%%%%%%%%%%%%%%%%%%%%%%%%%%%%%%%%%%%%%%%%%%%%
%%%%%%%%%%%%%%%%%%%%%%%%%%%%%%%%%%%%%%%%%%%%%%%%%%
%%%%%%%%%%%%%%%%%%%%%%%%%%%%%%%%%%%%%%%%%%%%%%%%%%

\section{Schur duality in type A}
In this section, we prove the Schur duality for $\mathfrak{sl}_n(\C)$ using the same methods that we will use to prove the Howe duality in type $C$. The root system associated to $\mathfrak{sl}_n(\C)$ is of type $A_{n-1}$ and we keep the notations of Example \ref{type_A} and Section \ref{settings}. 

\medskip

For all integers $n,m\in\Z_{\geq 1}$, we denote by $\cP_n$ the set of partitions 
of the form $(\la_1,\ldots,\la_n)$,
and by $\cP_{n,m}$ the set of partitions in $\cP_n$ such that $\la_1\leq n$
(i.e. the Young diagram of a partition in $\cP_{n,m}$ is included in a rectangle with sides $n\times m$). 
For any partition $\la\in \cP_{n,m}$, we denote by $(\la_1,\ldots,\la_n)$ the parts of~$\la$. The conjugate partition is denoted by $\la'$ and $\la'= (\la'_1,\ldots,\la'_m)$. Note that $\la\in \cP_{n,m}$ if and only if $\la'\in \cP_{m,n}$.
 
\medskip
  
We have seen in Example \ref{type_A}  that the set of dominant weights for a root system of type $A_{n-1}$ is in bijection with $\cP_{n-1}$. We will freely identify those sets. 
 
\medskip

For $N\in\Z_{\geq 1}$ and  $\gamma,\lambda\in \cP_N$ we write $K^{A_N}_{\gamma,\lambda}$ for the dimension of the $\gamma$-weight space in the irreducible $\mathfrak{sl}_N(\C)$-module $V(\lambda)$ of highest weight $\lambda$. For any $\mathfrak{sl}_N(\C)$-module $M$, we write $[M:V(\lambda)]$ for the multiplicity of $V(\lambda)$ in~$M$.

\begin{Th}[Schur duality]\label{schur_duality}
Let $\mu\in \cP_{n,m}$ and let $\mu' = (\mu'_1,\ldots,\mu'_m)$. For all $\la\in \cP_{n-1}$, we have
$$\left[{\sf \Lambda}^{\mu'_1}(\C^{n})\otimes \cdots \otimes {\sf \Lambda}^{\mu'_m}(\C^{n}):V(\la)\right] = K^{A_m}_{\la',\mu'}$$
where ${\sf \Lambda}^{k}(\C^{n})$ is the $k$-th exterior power of the natural representation $\C^n$ of $\mathfrak{sl}_n(\C)$.
\end{Th}

The character of ${\sf \Lambda}^{k}(\C^{n})$ is $\se_k(x_1,\ldots,x_n)$ and so the character of the tensor product 
${\sf \Lambda}^{\mu'_1}(\C^{n})\otimes \ldots \otimes {\sf \Lambda}^{\mu'_m}(\C^{n})$ is  $\se_{\mu'_1}\ldots\se_{\mu'_m}$. If we define $u_{\la,\mu'}$ by the relation
$$\se_{\mu'_1}\ldots\se_{\mu'_m} = \sum_{\la\in \cP_{n-1,m}} u_{\la,\mu'} {\sf s}_\la$$
the theorem states that $u_{\la,\mu'}=K^{A_m}_{\la',\mu'}$ for all $\la\in \cP_{n-1,m}$. 

\medskip

For $\beta = (\beta_1,\ldots,\beta_m)\in \Z^m$ we define the matrix $\sV(\beta)$ by 
$$[\sV(\beta)]_{i,j} = \se_{\beta_i + j - i}$$
where $\se_k = 0$ whenever $k\notin \{0,\ldots,n\}$. We set $\sv_\beta = \det(\sV(\beta))$. More explicitely we have:
 $$\sv_{\beta}  = \left|
\begin{array}{ccccccccc}
\se_{\beta_1} & \se_{\beta_1+1}   &\ldots & \se_{\beta_1+m-1}   \\
\se_{\beta_2 -1} &\se_{\beta_2 }&\ldots & \se_{\beta_2+m-2}   \\
\vdots & \vdots & \ddots & \vdots\\
\se_{\beta_m - m +1} &\se_{\beta_m - m +2}  &\ldots & \se_{\beta_m}  \\ 
\end{array}\right|.
$$
The well-known Jacobi-Trudi formula \cite[Appendix A]{FH91} tells us that for all $\la \in \cP_{n-1}$, we have ${\sf s}_{\la} = \sv_{\la'}$.
 
\begin{Rem}
\label{simplify}
If $\beta = (n,\ldots,n,\beta_k,\ldots,\beta_1) \in \Z^m$ then since $\se_{n}(x_1,\ldots,x_n) = x_1\cdots x_n = 1$ in $\Z[P]^W$ we have 
$$\sv_{\beta} = \sv_{\beta^\dag} \qu{where} \beta^\dag = (\beta_k,\ldots,\beta_1,0,\ldots,0)\in \Z^m.$$
\end{Rem}
\medskip

Given $\beta\in \Z^m$ we set $\sx^\beta = x_1^{\beta_1}\cdots x_m^{\beta_m}$ and $\se_{\beta} = \se_{\beta_1}\cdots \se_{\beta_m}$. We define $\sE$ to be the linear map
$$\begin{array}{ccccc}
\sE : & \Z[P] &\rightarrow & \Z[P]^W\\
&\sx^\beta &\mapsto & \se_{\beta} \end{array}$$
This map can be extended to the set of formal series $\Z[[P]]$. Indeed, there are only finitely many $\beta\in \Z^m$ such that $\se_\beta\neq 0$. The map $\sE$ satisfies the following useful lemma.

\begin{Lem}
\label{Lem_useful} Let $\beta$ and $\gamma$ in $\mathbb{Z}^{m}$ be such that there exists an integer $1\leq s\leq m$ with $\beta_{i}=0$ for $s+1\leq i\leq m$ and
$\gamma_{i}=0$ for $1\leq i\leq s$. Then \footnote{The map $E$ is not a morphism of algebras, the equality only holds when the variables are separated in a given monomial.}
\[
\sE(\sx^{\beta+\gamma})=\sE(\sx^{\beta}\cdot \sx^{\gamma})=\sE(\sx^{\beta})\cdot
\sE(\sx^{\gamma}).
\]

\end{Lem}

\begin{proof}
The equality $\sE(\sx^{\beta+\gamma})=\sE(\sx^{\beta}\cdot \sx^{\gamma})$ is clear since
we have $\sx^{\beta+\gamma}=\sx^{\beta}\cdot \sx^{\gamma}$. On the one hand,
we have 
$$\sE(\sx^{\beta+\gamma})=\se_{\beta_{1}}\cdots \se_{\beta_{s}}\cdot \se_{\gamma_{s+1}}\cdots\se_{\gamma_{m}}$$ by the hypothesis
on $\beta$ and $\gamma$.\ On the other hand, the same hypothesis implies that
$$\sE(x^{\beta})=\se_{\beta_{1}}\cdots \se_{\beta_{s}}\qu{and} \sE(\sx^{\gamma
})=\se_{\gamma_{s+1}}\cdots\se_{\gamma_{m}}$$ since $\se_{\beta_{i}}=1$
for any $s+1\leq i\leq m$ and $\se_{\gamma_{i}}=1$ for any $1\leq i\leq s$. This
yields the desired equality.
\end{proof}

We set 
$$\Delta^A_m = \prod_{\al\in R^+} (1 - e^{\al}) = \prod_{1\leq i < j\leq n}(1-\dfrac{x_i}{x_j}).$$
\begin{Prop}
\label{delta_A}
For all $\beta\in \Z^m$ we have $\sE(\Delta^A_m\cdot \sx^{\beta}) = \sv_{\beta}$.
\end{Prop}
\begin{proof}
Let $\beta = (\beta_1,\ldots,\beta_m)$. We have 
\begin{align*}
\left|
\begin{array}{ccccccccc}
x_1^{\beta_1} & x_1^{\beta_1+1}   &\ldots & x_1^{\beta_1+m-1}   \\
x_2^{\beta_2 -1} &x_2^{\beta_2 }&\ldots & x_2^{\beta_2+m-2}   \\
\vdots & \vdots & \ddots & \vdots\\
x_m^{\beta_m - m +1} &x_m^{\beta_m - m +2}  &\ldots & x_m^{\beta_m}  \\ 
\end{array}\right|
&= x_1^{\beta_1}\ldots x_m^{\beta_m-m+1}\cdot 
\left|
\begin{array}{ccccccccc}
1 & x_1   &\ldots & x_1^{m-1}   \\
1 & x_2   &\ldots & x_2^{m-1}   \\
\vdots & \vdots & \ddots & \vdots\\
1 & x_m   &\ldots & x_m^{m-1}   \\
\end{array}\right|\\
&= x_1^{\beta_1}\ldots x_m^{\beta_m-m+1}\cdot \prod_{1\leq i<j \leq m} (x_j - x_i)\\
&= x_1^{\beta_1}\ldots x_m^{\beta_m}\cdot \prod_{1\leq i<j \leq m} (1 -\frac{x_i}{x_j})\\
&= \Delta_m^A\cdot \sx^\beta.
\end{align*}
By expanding the determinant, we get
\begin{align*}
\sE(\Delta^A_m\cdot \sx^{\beta}) &= \sE\left(\sum_{\sigma\in {\frak S}_m} x_{1}^{\beta_{1} - (\sigma(1) - 1)}\ldots x_{m}^{\beta_{m} - (\sigma(m) - m)}   \right)\\
&= \sum_{\sigma\in {\frak S}_m} \sE(x_{1}^{\beta_{1} - (\sigma(1) - 1)})\ldots \sE(x_{m}^{\beta_{m} - (\sigma(m) - m)} )\\
&= \sum_{\sigma\in {\frak S}_m} \se_{\beta_{1} - (\sigma(1) - 1)}\ldots \se_{\beta_{m} - (\sigma(m) - m)}\\
&= \sv_\beta
\end{align*}
as required. 
\end{proof}
Recall the definition of the dot action in \Cref{settings}. 
\begin{Prop}
\label{straight_A}
Let $\beta\in \Z^m$. For all $w\in W$, we have $\sv_{w\circ \beta} = \eps(w)\sv_{\beta}$. Further either $\sv_\beta = 0$ or there exists a partition $\gamma\in \cP_{m,n-1}$ and $w\in W$ such that $\sv_\beta = \eps(w)\sv_{\gamma}$.
\end{Prop}
\begin{proof}
We prove the first assertion. To do so, it is enough to show that for all $i\in \{1,\ldots,n-1\}$, we have $\sv_{s_i\circ \beta} = -\sv_{\beta}$.  The matrix $\sV(s_i\circ \beta)$ only differs from $\sV(\beta)$ on rows $i$ and $i+1$. Let $\beta =(\beta_1,\ldots,\beta_m)$. We have 
\begin{center}
\begin{minipage}{0.45\linewidth}
\begin{align*}
[\sV(s_i\circ \beta)]_{i,j} &= \se_{(s_i\circ \beta)_i-i+j}\\
&= \se_{\beta_{i+1}-1-i+j}\\
&= \se_{\beta_{i+1}-(i+1)+j}\\
& = [\sV(\beta)]_{i+1,j} 
\end{align*}
\end{minipage}
\begin{minipage}{0.45\linewidth}
\begin{align*}
[\sV(s_i\circ \beta)]_{i+1,j} &= \se_{(s_i\circ \beta)_{i+1}-(i+1)+j}\\
&= \se_{\beta_{i}+1-(i+1)+j}\\
&= \se_{\beta_{i+1}-i+j}\\
& = [\sV(\beta)]_{i,j} 
\end{align*}
\end{minipage}
\end{center}
It follows that 
$$\sv_{s_i\circ \beta} = \det(\sV(s_i\circ \beta)) = - \det(\sV(\beta)) =  - \sv_{\beta}.$$
as required for the first assertion. 

\medskip

We prove the second assertion. Let $\beta\in \Z^m$ be such that $\sv_{\beta} \neq 0$. Using the first relation and the fact that the dot action is simply the action of the Weyl group translated by $-\rho$, we see that there exists $w\in W$ such that $\gamma = w\circ \beta$ lies in the chamber $t_{-\rho}(\cC_0)$. This means that for all $\al_i=\eps_i-\eps_{i+1}\in \Delta$, we have 
$$\scal{\gamma}{\al_i} \geq \scal{-\rho}{\al_i} = -1.$$
If there exists $\al_i$ such that $\scal{\gamma}{\al}=-1$, then $s_i\circ \gamma = \gamma$ and this forces $\sv_{\gamma}$ and $\sv_{\beta}$ to be $0$. This is impossible since we assumed that $\sv_{\beta}\neq 0$ hence we must have $\scal{\gamma}{\al_i}\geq 0$ for all $i$, i.e. $\gamma_i \geq \gamma_{i+1}$. Next we must have $\gamma_1\leq n$ otherwise the first row of $\sV(\gamma)$ is $0$ and  $\gamma_m\geq 0$  otherwise the last row of $\sV(\gamma)$ is $0$. Finally by Remark~\ref{simplify}, we can remove all the parts of $\gamma$ that are equal to $n$ to obtain a partition $\gamma^\dag$ such that $\gamma^\dag_1<n$ and $\gamma^\dag_0\geq 0$ and $\sv_{\beta} = \pm\sv_{\gamma^\dag}$, as required.  
\end{proof}
We are now ready to prove the Schur duality.

\medskip

\textit{Proof of  \Cref{schur_duality}}. 
Let $\mu\in\cP_{n,m}$.  We have
\begin{align*}
 \sE(\sx^{\mu'} ) &=  \sE(\Delta^A_m \cdot\dfrac{1}{\Delta^A_m}\cdot  \sx^{\mu'}) \\
 &= \sE(\Delta^A_m \cdot\sum_{\beta\in \Z^m} \cP(\beta)\sx^{\beta+\mu'}) \tag{by definition of ${\cal P}$}\\
 &= \sum_{\beta\in \Z^m} \cP(\beta)\sE(\Delta^A_m \cdot\sx^{\beta+\mu'}) \\
  &= \sum_{\beta\in \Z^m} \cP(\beta)\sv_{\beta+\mu'} \tag{by Proposition \ref{delta_A}}\\
  &= \sum_{\gamma\in \cP_{m,n-1}} \sum_{w\in W}\eps(w)\mathcal{P}(w\circ \gamma - \mu' ) \sv_{\gamma} \tag{by Proposition \ref{straight_A}}\\
    &= \sum_{\gamma\in \cP_{m,n-1}} K^{A_m}_{\gamma,\mu'} {\sf s}_{\gamma'}\tag{by \Cref{Kostant}}\\
        &= \sum_{\la\in \cP_{n-1,m}} K^{A_m}_{\la',\mu'} {\sf s}_{\la}.
\end{align*} 
Since $\sE(\sx^{\mu'}) =\se_{\mu'}$
we get the result.  
\hfill
$\square$

%%%%%%%%%%%%%%%%%%%%%%%%%%%%%%%%%%%%%%%%%%%%%%%%%%
%%%%%%%%%%%%%%%%%%%%%%%%%%%%%%%%%%%%%%%%%%%%%%%%%%
%%%%%%%%%%%%%%%%%%%%%%%%%%%%%%%%%%%%%%%%%%%%%%%%%%
%%%%%%%%%%%%%%%%%%%%%%%%%%%%%%%%%%%%%%%%%%%%%%%%%%
%%%%%%%%%%%%%%%%%%%%%%%%%%%%%%%%%%%%%%%%%%%%%%%%%%
%%%%%%%%%%%%%%%%%%%%%%%%%%%%%%%%%%%%%%%%%%%%%%%%%%

\section{Howe duality in type $C$}
\label{Sec_HoweDuality}
In this section, we prove the Howe duality for $\mathfrak{sp}_{2n}(\C)$. The root system associated to $\mathfrak{sp}_{2n}(\C)$ is of type $C_{n}$ and we keep the notation of Example \ref{type_C} and Section \ref{settings}. Recall that the set of dominant weights of $\mathfrak{sp}_{2n}(\C)$ is in bijection with $\cP_n$. We will freely identify those two sets.  

\medskip

Let $V$ be the Euclidean space of dimension $m$ with basis $(\eps_1,\ldots,\eps_m)$. We define the involution $\sI$ on $V$ by
$$\sI(\beta_1,\ldots,\beta_m) = (-\beta_m,\ldots,-\beta_1).$$

\begin{Def}
\label{def_hat}
Let $\beta\in \cP_{n,m}$. We define the partition $\wh{\beta}\in \cP_{m,n}$ by
$$\wh{\beta} := \sI(\beta') + n\cdot (1,\ldots,1).$$
\end{Def}
The partition $\wh{\beta}$ is obtained by taking the conjugate of the complement of $\beta$ in the rectangle $n\times m$. Note that the map $\wh{\ \cdot\ }$ depends on the integers $n$ and $m$. It sends a weight of $C_n$ to a weight of $C_m$.  
\begin{Exa}
Let $\beta = (5,4,2,1)\in 4\times 5$. Then we have $\beta' = (4,3,2,2,1)$ and
$$\wh{\beta} = \sI(\beta') + 4\cdot (1,1,1,1,1) = (3,2,2,1,0).$$
In the figure below, we represent the partition $\wh{\beta}$ (in green) as the conjugate complement of the partition $\beta$ (in white) in the rectangle $4 \times 5$:
$$
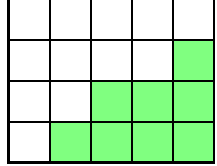
\ydiagram[*(white)]{5,4,2,1}*[*(green!50)]{5,5,5,5}
$$
\end{Exa}
The map $\wh{\ \cdot \ }$ will play in the Howe duality a role analogue to the conjugation in the Schur duality. For $N\in\Z_{\geq 1}$ and $\gamma,\lambda\in \cP_N$, we write $K^{C_N}_{\gamma,\lambda}$ for the dimension of the
$\gamma$-weight space in the irreducible representation~$V(\lambda)\in \mathfrak{sp}_{2N}(\C)$ of highest weight $\lambda$.

\begin{Th}[Howe Duality \cite{H95}]
\label{Thm_Howe_duality}
Let $\mu\in \cP_{n,m}$ and let $\mu' = (\mu'_1,\ldots,\mu'_m)$. For all $\la\in \cP_n$, we have
$$\left[{\sf \Lambda}^{\mu'_1}(\C^{2n})\otimes \ldots \otimes {\sf \Lambda}^{\mu'_m}(\C^{2n}):V(\la)\right] = K^{C_m}_{\wh{\la},\wh{\mu}}$$ 
where ${\sf \Lambda}^{k}(\C^{2n})$ is the $k$-th exterior power of the natural representation $\C^{2n}$ of $\mathfrak{sp}_{2n}(\C)$.
\end{Th}
Following \cite[Section 24.2]{FH91}, the character of ${\sf \Lambda}^{k}(\C^{2n})$ is $\se_{k}(x_1,\ldots,x_n,x_1^{-1},\ldots,x_n^{-1})$. If we define the coefficients $u$ by setting (we omit the variables)
 $$ \se_{\mu'_1}\ldots \se_{\mu'_m} = \sum_{\la\in \cP_{n,m}} u_{\la,\mu'} {\sf s}_\la$$
then the theorem states that $u_{\la,\mu'}=K^{C_m}_{\wh{\la},\wh{\mu}}$ for all $\la\in \cP_{n,m}$.

\medskip

For $\beta = (\beta_1,\ldots,\beta_m)\in \Z^m$ we define the matrix $\sV(\beta)$ by 
$$[\sV(\beta)]_{i,j} = \se_{\beta_i - (i+j)} - \se_{\beta_i - (i-j)}$$
where $\se_k = 0$ whenever $k\notin \{0,\ldots,2n\}$. We set $\sv_{\beta} = \det(\sV(\beta))$. More explicitely we have
$$\sv_\beta = 
\left|
\begin{array}{ccccccccc}
\se_{\beta_1} -\se_{\beta_1-2} & \se_{\beta_1+1} - \se_{\beta_1-3} &\ldots & \se_{\beta_1 + m-1} - \se_{\beta_1-m -1}   \\
\se_{\beta_2-1} - \se_{\beta_2-3} & \se_{\beta_2} - \se_{\beta_2-4} &\ldots & \se_{\beta_2 + m-2} - \se_{\beta_2-m -2}   \\
\vdots & \vdots & \ddots & \vdots\\
\se_{\beta_m-m+1} - \se_{\beta_m-m-1} & \se_{\beta_m-m+2} - \se_{\beta_m-m-2} &\ldots & \se_{\beta_m} - \se_{\beta_m-2m}   \\
\end{array}
\right|.
$$
The Jacobi-Trudi formula in type $C_n$ states that for all $\la\in\cP_{n,m}$, we have $\fs_{\la} = \sv_{\la'}$. 

\medskip

Given $\beta\in \Z^m$ we set $\sx^\beta = x_1^{\beta_1}\ldots x_n^{\beta_n}$ and $\se_{\beta} = \se_{\beta_1}\ldots \se_{\beta_n}$. We then define $\sE$ to be the linear application:
$$\begin{array}{ccccc}
\sE : & \Z[P] &\rightarrow & \Z[P]^{W}\\
&\sx^\beta &\mapsto & \se_{\beta} \end{array}$$
It is not hard to check that the map $\sE$ can be extended to $\Z[[P]]$ and that it satisfies Lemma \ref{Lem_useful}.

\medskip

The involution $\sI$ on $V$ induces a ring involution $\Z[P]$ that maps $e^{\beta}$ to $e^{\sI(\beta)}$ for all $\beta\in P$.  In particular it maps $x_i=e^{\eps_i}$ to $x^{-1}_{m-i}=e^{-\eps_{m-i}}$. We will still denote this involution by $\sI$. 
We set 
$$\Delta^C_m= \sI\left(\prod_{\al\in R^+}(1-e^\al)\right)= \prod_{1\leq i < j \leq m} \left(1-\dfrac{x_i}{x_j} \right)\prod_{1\leq i \leq  j \leq m} \left(1-\dfrac{1}{x_ix_j} \right).$$

\begin{Prop}
\label{delta_C}
For all $\beta\in \Z^m$, we have $\sE(\Delta^C_m\cdot \sx^{\beta}) = \sv_{\beta}$.
\end{Prop}
\begin{proof}
Let $\beta = (\beta_1,\ldots,\beta_m)$. We have
\begin{align*}
\det(x_i^{\beta_i - i + j} -x_i^{\beta_i - i - j} ) &= \prod_{i=1}^m(x_i^{\beta_i - i + 1} -x_i^{\beta_i - i - 1})
 \left|
\begin{array}{ccccccccc}
1& x_1 + x_1^{-1} &\ldots &  x_1^{m-1} + x_1^{-m+1}  \\
1& x_2 + x_2^{-1} &\ldots &  x_2^{m-1} + x_2^{-m+1}  \\
\vdots & \vdots & \ddots & \vdots\\
1& x_m + x_m^{-1} &\ldots &  x_m^{m-1} + x_m^{-m+1}  \\
\end{array}
\right|\\
&= \prod_{i=1}^m(x_i^{\beta_i - i + 1} -x_i^{\beta_i - i - 1})
 \left|
\begin{array}{ccccccccc}
1& x_1 + x_1^{-1} &\ldots &  (x_1 + x_1^{-1})^{m-1}  \\
1& x_2 + x_2^{-1} &\ldots &  (x_2 + x_2^{-1})^{m-1}  \\
\vdots & \vdots & \ddots & \vdots\\
1& x_m + x_m^{-1} &\ldots &  (x_m + x_m^{-1})^{m-1}  \\
\end{array}
\right|\\
&=  \prod_{i=1}^m(x_i^{\beta_i - i + 1} -x_i^{\beta_i - i - 1}) \prod_{1\leq i<j\leq m}(x_j + x_j^{-1} - x_i - x_i^{-1})\\
&=  \prod_{i=1}^m(x_i^{\beta_i - i + 1} -x_i^{\beta_i - i - 1}) \prod_{1\leq i<j\leq m}x_j(1-\dfrac{x_i}{x_j})(1-\dfrac{1}{x_ix_j})\\
&=  \prod_{i=1}^m(x_i^{\beta_i - i + 1} -x_i^{\beta_i - i - 1}) {x_2\ldots x_m^{m-1}}\prod_{1\leq i<j\leq m}(1-\dfrac{x_i}{x_j})(1-\dfrac{1}{x_ix_j})\\
&=  \prod_{i=1}^m(x_i^{\beta_i} -x_i^{\beta_i - 2}) \prod_{1\leq i<j\leq m}(1-\dfrac{x_i}{x_j})(1-\dfrac{1}{x_ix_j})\\
&=  \sx^{\beta}\prod_{i=1}^m(1-\dfrac{1}{x_i^2})\prod_{1\leq i<j\leq m}(1-\dfrac{x_i}{x_j})(1-\dfrac{1}{x_ix_j})\\
&=\Delta^C_m\cdot \sx^{\beta}
\end{align*}
as required. By expanding the determinant, we get
\begin{align*}
\sE(\Delta^C_m\cdot \sx^{\beta}) &= \sE\left(\sum_{\sigma\in {\frak S}_m} \prod_{i=1}^m x_i^{\beta_i - (\sigma(i) - i)} - x_i^{\beta_i - (\sigma(i) - i) -2}   \right)\\
&=\sum_{\sigma\in {\frak S}_m} \sE\left(\prod_{i=1}^m x_i^{\beta_i - (\sigma(i) - i)} - x_i^{\beta_i - (\sigma(i) - i) -2}   \right)\\
&=\sum_{\sigma\in {\frak S}_m} \prod_{i=1}^m  \sE\left(x_i^{\beta_i - (\sigma(i) - i)} - x_i^{\beta_i - (\sigma(i) - i) -2}   \right)\tag{by Lemma \ref{Lem_useful}}\\
&=\sum_{\sigma\in {\frak S}_m} \prod_{i=1}^m  \se_{\beta_i - (\sigma(i) - i)} - \se_{\beta_i - (\sigma(i) - i) -2}\\
&= \sv_\beta.
\end{align*}
\end{proof}
In the proof of the Schur duality, we used the trick 
$$\sE(\sx^{\mu'}) = \sE(\Delta_m^A\cdot\dfrac{1}{\Delta_m^A} \sx^{\mu'})$$
and we expanded the inverse of $\Delta_m^A$ with the Kostant partition function. Using the same idea here involves some twisted Kostant partition function.
\begin{Def}(\cite[Lemma 3.3.1]{Le06})
The element $\Delta^C_m$ is invertible in the ring of formal Laurent series in the variables $(x_i^{\pm 1})_{1\leq i\leq n}$ and 
we can define a partition function associated to $\Delta^C_m$ by setting 
$$\dfrac{1}{\Delta^C_m} = \sum_{\beta\in \Z^m} \widetilde{\mathcal{P}}(\beta) \sx^{\beta}.$$
\end{Def}

\begin{Lem}
\label{PP_tilde}
We have  $\widetilde{\mathcal{P}}(\beta) =\mathcal{P}(\sI(\beta))$ where ${\cal P}$ is the usual Kostant partition function.
\end{Lem}
\begin{proof}
Using the fact that $\sI$ is an involution we get
$$\dfrac{1}{\Delta^C_m} = \sI\left(\dfrac{1}{\underset{\al\in R_+}{\prod}(1-e^\al))}\right) =\sI(\sum_{\beta\in \Z^m}\cP(\beta)\sx^\beta) = \sum_{\beta\in \Z^m} \cP(\beta)\sx^{\sI(\beta)} = \sum_{\beta\in \Z^m} \cP(\sI(\beta))\sx^{\beta} $$
hence the result. 
\end{proof}

The next step in our proof of the Schur duality in type $A$ was to show that the determinant $\sv_\beta$ satisfied the relation $\sv_{w\circ \beta} = \eps(w)\sv_{\beta}$ for all $w\in {\frak S}_m$. In type $C_m$, we need to modify the dot action to obtain such a relation. More precisely, we are looking for $\delta\in P$ such that we have for all $w\in W$:
$$\sv_{w\circ \beta} = \eps(w)\sv_{\beta} \qu{where} w\circ \beta = w(\beta+\delta) - \delta.$$
It will turn out that there is a unique $\delta$ satisfying this relation. The proof is based on the following symmetry.
\begin{Prop}
We have $\se_{n+k} = \se_{n-k}$ for all $k\in \Z$ where $\se_k = 0$ if $k\notin \{1,\ldots,2n\}$.
\end{Prop}
Recall that the $m$-th fundamental weight in a root system of type  $C_m$ is $\om_m = (1,\ldots,1)$ and that the half sum of positive roots is $\rho_m = (m,\ldots,1)$.

\begin{Th}
\label{straightening_C}
Let $\delta = \sI(\rho) - m\cdot \om_m\in \Z^m$ and define a "dot action" of $W$ on $\Z^m$ by setting
$$w\circ \beta =w(\beta+\delta) - \delta.$$
We have
\begin{enumerate}
\item For all $\beta\in \Z^m$, we have $\sv_{w\circ \beta} = \eps(w)\sv_{\beta}$. 
\item If $\sv_\beta\neq 0$, there exists a partition $\gamma\in\cP_{m,n}$ and $w\in W$ such that  $w\circ \beta = \gamma$ and $\sv_\beta =\eps(w) \sv_\gamma$.
\end{enumerate}
\end{Th}

\begin{proof} We have $\delta= (-n-1,\ldots,-n-m)$. Recall that the Weyl group $W$ is generated by the set $\{s_1,\ldots,s_{n-1},s_n\}$ where $s_i$ acts on $V$ by permuting the $i$ and the $i+1$-th coordinates and $s_n$ acts by changing  the sign of the $n$-th coordinate.
\begin{enumerate}
\item To prove this assertion, it is enough to show that 
$$\sv_{s_i\circ \beta} = -\sv_{\beta} \text{ for all $1\leq i \leq n-1$}\qu{and} \sv_{s_{n}\circ \beta}   = -\sv_{\beta}.$$
Let $\beta = (\beta_1,\ldots,\beta_m)\in \Z^m$ and $i\in \{1,\ldots,m-1\}$. We have 
\begin{align*}
s_i\circ\beta &= s_i(\beta+\delta) - \delta\\
 &= (\beta_1+\delta_1,\ldots,\beta_{i+1}+\delta_{i+1},\beta_{i}+\delta_{i},\ldots,\beta_m+\delta_m) - (\delta_1,\ldots,\delta_m)\\
  &= (\beta_1,\ldots,\beta_{i+1}+\delta_{i+1}-\delta_i,\beta_{i}+\delta_{i}-\delta_{i+1},\ldots,\beta_m)\\
    &= (\beta_1,\ldots,\beta_{i+1} - 1,\beta_{i}+1,\ldots,\beta_m).
\end{align*}
The matrix $\sV(s_i\circ \beta)$ only differs from the matrix  $\sV(\beta)$ on rows $i$ and $i+1$. We have 
\begin{center}
\begin{minipage}{0.45\linewidth}
\begin{align*}
[\sV(s_i\circ \beta)]_{i,j} &= \se_{[s_i\circ \beta]_i - (i+j)} - \se_{[s_i\circ \beta]_i - (i-j)}\\
&= \se_{\beta_{i+1} - 1 - (i+j)} - \se_{\beta_{i+1} - 1 - (i-j)}\\
&= \se_{\beta_{i+1} - \left((i+1)  + j\right)} - \se_{\beta_{i+1} -((i+1)-j)}\\
& = [\sV(\beta)]_{i+1,j} 
\end{align*}
\end{minipage}
\begin{minipage}{0.45\linewidth}
\begin{align*}
[\sV(s_i\circ \beta)]_{i+1,j} &= \se_{[s_i\circ \beta]_{i+1} - (i+1+j)} - \se_{[s_i\circ \beta]_{i+1} - (i+1-j)}\\
&= \se_{\beta_i+1 - (i+1+j)} - \se_{\beta_i+1 - (i+1-j)}\\
&= \se_{\beta_i - (i+j)} - \se_{\beta_i - (i-j)}\\
& = [\sV(\beta)]_{i,j} 
\end{align*}
\end{minipage}
\end{center}
Hence the result on the determinant. 

\medskip

Next we have 
\begin{align*}
s_m\circ\beta &= s_m(\beta+\delta) - \delta\\
 &= (\beta_1+\delta_1,\ldots,\beta_{m-1}+\delta_{m-1},-\beta_m-\delta_m) - (\delta_1,\ldots,\delta_m)\\
  &= (\beta_1,\ldots,\beta_{m-1},-\beta_m-2\delta_m)\\
  &= (\beta_1,\ldots,\beta_{m-1},-\beta_m +2n +2m).
\end{align*}
The matrix $\sV(s_n\circ \beta)$ only differs from the matrix  $\sV(\beta)$ on rows $m$ and we have 
\begin{align*}
[\sV(s_m\circ \beta)]_{m,j} &= \se_{[s_n\circ \beta]_m - (m+j)} - \se_{[s_i\circ \beta]_m - (m-j)}\\
&= \se_{-\beta_m +2n +2m - (m+j)} - \se_{-\beta_m +2n +2m - (m-j)}\\
&= \se_{n  + (n -\beta_m  +m - j)} - \se_{n + (n -\beta_m +m+j)}\\
&= \se_{n  - (n -\beta_m  +m - j)} - \se_{n - (n -\beta_m +m+j)}\tag{by the previous proposition}\\
&= \se_{\beta_m   - m + j} - \se_{\beta_m - m -j}\\
&= -[\sV(\beta)]_{m,j} 
\end{align*}
hence the result on the determinant. 

\item
Let $\beta\in \Z^m$ be such that $\sv_\beta \neq 0$. The "dot action" is simply the action of $W$ translated by $-\delta$.
This action is transitive on the set of Weyl chambers centered at $\delta$. The set of Weyl chambers centered at $-\delta$ is the set of Weyl chambers centered at $0$ translated by $-\delta$. Now the set of Weyl chambers centered at 0 is parametrised by the set of simple system of the form $w\Delta$ where $w\in W$. More precisely, the Weyl chamber $w({\cal C}_0)$ is
$$w({\cal C}_0) = \{x\in V\mid \scal{x}{\al} > 0 \text{ for all $\al\in w(\Delta)$}\}.$$
We have $\sI\in W$ and 
$$\sI(\Delta) = \{\sI(\al_1),\ldots,\sI(\al_{m-1}),\sI(2\eps_m)\} = \{\al_1,\ldots,\al_{m-1},-2\eps_1\}.$$
As a consequence the Weyl chamber $t_{-\delta}(\sI({\cal C}_0))$ is defined by 
$$x\in t_{-\delta}(\sI({\cal C}_0)) \Leftrightarrow  \scal{x}{\al^\vee} > \scal{-\delta}{\al^\vee} \text{ for all $\al\in \sI(\Delta)$}.$$
By transitivity of the dot action, there exists $w\in W$ such that $\gamma = w\circ \beta\in t_{-\delta}(\sI({\cal C}_0))$. Further, since we assumed that $\sv_\beta\neq 0$, $\gamma$ cannot lie on the wall of the chamber $t_{-\delta}(\sI({\cal C}_0))$.
The equations above for the root $\al_i\in \sI(\Delta)$ and $-2\eps_1\in \sI(\Delta)$ yield (using $\al_i^\vee = \al_i$ and $(-2\eps_1)^\vee = -\eps_1$)
$$\scal{\gamma}{\al_i}>\scal{-\delta}{\al_i} = -\delta_i + \delta_{i+1} = -1\qu{and} \scal{\gamma}{-\eps_1}>\scal{-\delta}{-\eps_1} = -(n+1)$$ 
in other words 
$$\gamma_i - \gamma_{i+1}\geq 0\qu{and} \gamma_1\leq n$$ 
which was what we were looking for, since $\sv_\beta = \eps(w)\sv_{\gamma}$. 
\end{enumerate}
 \end{proof}

We are now ready to prove the Howe duality.

\medskip

\textit{Proof of \Cref{Thm_Howe_duality}.} 
Let $\mu\in\cP_{n,m}$.  First we have 
\begin{align*}
\sE(\sx^{\mu'} ) &= \sE\left(\Delta^C_m \cdot\dfrac{1}{\Delta^C_m}\cdot  \sx^{\mu'}\right) \\
&=\sE\left(\Delta^C_m \sum_{\beta\in \Z^m} \wP(\beta) \sx^{\beta+\mu'}\right)\\
&= \sum_{\beta\in \Z^m} \wP(\beta) \sE\left(\Delta^C_m\sx^{\beta+\mu'}\right)\\
&= \sum_{\beta\in \Z^m} \wP(\beta) \sv_{\beta+\mu'}\\
&= \sum_{\gamma\in \cP_{m,n}} \sum_{w\in W}\eps(w)\wP(w\circ\gamma - \mu' ) \sv_{\gamma} \\
&= \sum_{\gamma\in \cP_{n,m}} \sum_{w\in W}\eps(w)\wP(w\circ \la' - \mu' ) \fs_{\la}.
\end{align*}
This shows that 
$$u_{\la,\mu'} = \sum_{w\in W}\eps(w)\wP(w\circ \la' - \mu' ).$$ 
Next, using the fact that $\sI(\delta_{m,n}) = \rho + n\cdot \om_m$ and the equality $\wP(\beta) = \cP(\sI(\beta))$ we get
\begin{align*}
\sum_{w\in W}\eps(w)\wP(w\circ \la' - \mu' ) &=  \sum_{w\in W}\eps(w)\cP(\sI(w\circ \la') - \sI(\mu') )\\
 &=  \sum_{w\in W}\eps(w)\cP(\sI(w(\la'+\delta_{m,n}) - \delta_{m,n}) - \sI(\mu') )\\
 &=  \sum_{w\in W}\eps(w)\cP(w(\sI(\la')+\sI(\delta_{m,n})) - \sI(\delta_{m,n}) - \sI(\mu') )\\
 &=  \sum_{w\in W}\eps(w)\cP(w(\sI(\la')+\rho + n\cdot \om_m) - \rho - n\cdot \om_m - \sI(\mu') )\\
  &=  \sum_{w\in W}\eps(w)\cP(w(\sI(\la') + n\cdot \om_m + \rho)  -(\sI(\mu') + n\cdot \om_m) -\rho)\\
  &=  \sum_{w\in W}\eps(w)\cP(w(\widehat{\la} + \rho)  -\widehat{\mu}-\rho)\\
&= K^{C_m}_{\widehat{\la},\widehat{\mu}}
\end{align*}
as desired. 
\hfill
$\square$

%%%%%%%%%%%%%%%%%%%%%%%%%%%%%%%%%%%%%%%%%%%%%%%%%%
%%%%%%%%%%%%%%%%%%%%%%%%%%%%%%%%%%%%%%%%%%%%%%%%%%
%%%%%%%%%%%%%%%%%%%%%%%%%%%%%%%%%%%%%%%%%%%%%%%%%%
%%%%%%%%%%%%%%%%%%%%%%%%%%%%%%%%%%%%%%%%%%%%%%%%%%
%%%%%%%%%%%%%%%%%%%%%%%%%%%%%%%%%%%%%%%%%%%%%%%%%%
%%%%%%%%%%%%%%%%%%%%%%%%%%%%%%%%%%%%%%%%%%%%%%%%%%

\section{Combinatorial Howe duality}

\label{Sec_Howe_by_crystals}

\newcommand{\sC}{\mathscr{C}}
\newcommand{\cB}{\mathcal{B}}
\newcommand{\cK}{\mathcal{K}}
\newcommand\sesqui{1.5}

The goal of this section is to give a bijective proof of \Cref{Thm_Howe_duality}.
This is achieved by establishing in \Cref{comb_duality_C} a combinatorial duality between 
a certain set of tensor products of type $C_n$ columns on the one hand,
and a set of type $C_m$ tableaux called \textit{King tableaux} on the other hand. To make this duality consistent with the usual convention on the combinatorial objects that we shall need (tensor products of crystals and King tableaux) it will be convenient to use the following realisation of the root system of type $C_n$.

\medskip

Let $n\in\Z_{\geq 1}$ and consider the type $C_n$ alphabet 
$$\sC_n=\left\{ \overline{n} < \cdots < \overline{1} < 1<\cdots < n\right\}.$$
This enables us to realise the  root system of type $C_n$ by setting
$$
\left\{
\begin{array}{l}
\al_i = \eps_{\overline{i}}-\eps_{\overline{i+1}} \quad  \text{ for } i=1,\ldots, n-1
\\
\al_0 = 2\eps_{\overline{1}}.
\end{array}
\right.
$$
for the simple roots, where
$\eps_{\overline{i}} =-\eps_i$ for all $i=1,\ldots, n$,
and 
$$\om_i=\eps_{\overline{n}}+\cdots + \eps_{\overline{i+1}} \quad \text{ for } i=0,\ldots, n-1.$$

\medskip

\Yvcentermath1

A column of height $k$ on $\sC_n$ (also called a column of type $C_n$) 
is a subset $c$ of $\sC_n$ of cardinality $k$, which we
represent by the Young tableau of shape $(1^k)$ filled by the elements of $c$, increasing from top to bottom. We will write $|c| = k$.
For any column $c$ on $\sC_n$, and for all $i=1,\ldots, n$ 
let
$$N_i(c)=\left|\{ x\in c \mid  x\leq \overline{i} \text{ or } x\geq i\}\right|.$$

\begin{Def}\label{Def_admissible}
A column $c$ on $\sC_n$  is called \textit{$n$-admissible} if 
$N_i(c)\leq n-i+1$ for all $i=1,\ldots, n$.
\end{Def}

\begin{Exa}
The set $c=\{\overline{2},\overline{1},1,3\}=
\scriptsize
\gyoung(<\overline{2}>,<\overline{1}>,1,3)
$
is a column on $\sC_n$ for all $n\geq 3$, and we have $|c|=4$.
It is not $3$-admissible since
$N_1(c)= \left| \{ 3,\overline{2},\overline{1},1\}\right| =4 > 3 -1+1$.
However, $c$ is $n$-admissible for $n\geq 4$.
\end{Exa}

Let us recall the crystal structure on the set of columns
of a given height due to \cite[Section 4.3]{KN94}.
First, 
columns of height $1$ realise the crystal of the vector representation of type $C_n$ as follows:
\begin{center}
\begin{tikzpicture}
\node (a) at (0,0) {$\scriptsize\gyoung(<\overline{n}>)$};
\node (b) at (1.5,0) {$\cdots$};
\node (c) at (3,0) {$\scriptsize\gyoung(<\overline{1}>)$};
\node (d) at (4.5,0) {$\scriptsize\gyoung(<1>)$};
\node (e) at (6,0)  {$\cdots$};
\node (f) at (7.5,0) {$\scriptsize\gyoung(<n>)$};

\draw[->] (a) --  node[pos=0.5,above]{\tiny $n-1$} (b);
\draw[->] (b) --  node[pos=0.5,above]{\tiny $1$} (c);
\draw[->] (c) --  node[pos=0.5,above]{\tiny $0$} (d);
\draw[->] (d) --  node[pos=0.5,above]{\tiny $1$} (e);
\draw[->] (e) --  node[pos=0.5,above]{\tiny $n-1$} (f);
\end{tikzpicture}
\end{center}
where the arrow labeled by $i$ denotes the action of the Kashiwara crystal operator $ f_i$.
In other terms, since 
$$\wt\left({\scriptsize\gyoung(<x>)}\right)=\sum_{i=1}^{n}a_i\eps_{\overline{i}} \text{\quad where \quad}
a_i=
\left\{
\begin{array}{rl}
1 & \text{\quad if } x=\overline{i}
\\
-1 & \text{\quad if } x=i
\\
0 & \text{\quad  otherwise,}
\end{array}
\right.
$$
this realises the crystal $ B(\om_{n-1})$ of the fundamental representation $V(\om_{n-1})$. 
To get the crystal structure on any tensor power $B(\omega_{n-1})^{\otimes\ell}$, we use Kashiwara's tensor product rule. 
It works as follows. 
Each vertex
$b={\scriptsize\gyoung(<x_1>)\otimes\gyoung(<x_2>)\otimes\cdots\otimes\gyoung(<x_\ell>)}\in B(\omega_{n-1})^{\otimes\ell}$ can be identified with its word
$\w=x_{1}\cdots x_{\ell}$ of length $\ell$ on $\sC_n$.
Now,
label each letter of $\w$ in $\{\overline{i+1},i\}$ (resp. in
$\{\overline{i},i+1\}$) with a symbol $+$ (resp. $-$) and ignore the
others. 
Let $\w_{i}$ be the word in the symbols $+$ and $-$ so obtained.
Bracket recursively all possible $+-$ in $\w_{i}$ (forgetting the previously bracketed symbols).
Then $ f_i \w$ (resp. $ e_i \w$) is obtained by applying $ f_i$ (resp. $ e_i$)
to the letter of $\w$ which contributes as the 
leftmost unbracketed $+$ (resp. the rightmost unbracketed $-$).
If this symbol does not exist, we set $ f_i \w= 0$ (resp. $ e_i \w=0$).
To compute $ f_{0}\w$ (resp. $ e_0 \w$), one proceeds similarly
but this time by encoding only the letters $\overline{n}$ in $\w$ by a $+$ and
the letters $n$ by a $-$. 

\begin{Exa}
\label{Ex_crystalrule1} Let $n=2$ and $\w=\bar{1}2\bar{2}1221\bar{1}\bar{2}$. Then
$\w_{1}=--++--+-+$ and we get the bracketing $\w_1=--(+(+-)-)(+-)+$. This gives
$ f_{1} \w= \bar{1}2\bar{2}1221\bar{1}\bar{1}$ and $  e_1 \w =\bar{1}1\bar{2}1221\bar{1}\bar{2}$.
\end{Exa}

One can now interpret
any column $c=\scriptsize\gyoung(<x_1>,|\sesqui\vdts,<x_\ell>)$ of height $\ell$ as the element
$\scriptsize\gyoung(<x_1>)\otimes\cdots \otimes \gyoung(<x_\ell>) \in  B(\om_{n-1})^{\otimes \ell}$
In particular, for all $i=0, \ldots, n-1$, 
the column 
$c=\scriptsize\gyoung(<\overline{n}>,|\sesqui\vdts,<\overline{n\text{-}i}>)$ is a highest weight vertex of weight $\om_{n-i-1}$,
and therefore generates the crystal $ B(\om_{n-i-1})$ of the fundamental representation $V(\om_{n-i-1})$.
Note that the shape of $c$ is given by the partition $(1^{i+1})$, that is,
the coordinates of $\wt(c)=\om_{n-i-1}$ in the basis $(\eps_{\overline{n}},\ldots, \eps_{\overline{1}}).$
The following result is due to \cite[Section 4.5]{KN94}, see also \cite[Proposition 4.2.1]{Le07} for the reformulation using \Cref{Def_admissible}.

\begin{Th}
The vertices of $B(\om_{n-i-1})$ are the $n$-admissible columns.
\end{Th}

\begin{Exa} Take $n=3$ and $i=1$. We realise the fundamental crystal $ B(\om_1)$ by $n$-admissible columns of height $2$ as follows.
Note that only ${\scriptsize\gyoung(<\overline{3}>,<3>)}$ is not $n$-admissible.
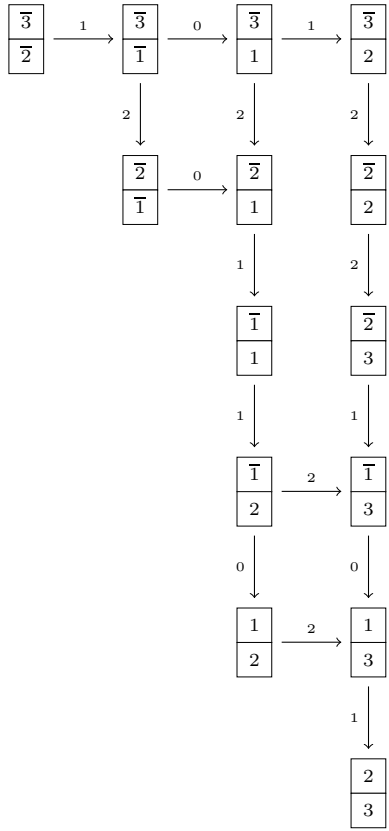
\begin{center}
\begin{tikzpicture}
\node (11) at (0,0) {$\scriptsize\gyoung(<\overline{3}>,<\overline{2}>)$};
\node (12) at (1.5,0) {$\scriptsize\gyoung(<\overline{3}>,<\overline{1}>)$};
\node (13) at (3,0) {$\scriptsize\gyoung(<\overline{3}>,<1>)$};
\node (14) at (4.5,0) {$\scriptsize\gyoung(<\overline{3}>,<2>)$};

\node (22) at (1.5,-2) {$\scriptsize\gyoung(<\overline{2}>,<\overline{1}>)$};
\node (23) at (3,-2) {$\scriptsize\gyoung(<\overline{2}>,<1>)$};
\node (24) at (4.5,-2) {$\scriptsize\gyoung(<\overline{2}>,<2>)$};

\node (33) at (3,-4) {$\scriptsize\gyoung(<\overline{1}>,<1>)$};
\node (34) at (4.5,-4) {$\scriptsize\gyoung(<\overline{2}>,<3>)$};

\node (43) at (3,-6) {$\scriptsize\gyoung(<\overline{1}>,<2>)$};
\node (44) at (4.5,-6) {$\scriptsize\gyoung(<\overline{1}>,<3>)$};

\node (53) at (3,-8) {$\scriptsize\gyoung(<1>,<2>)$};
\node (54) at (4.5,-8) {$\scriptsize\gyoung(<1>,<3>)$};

\node (64) at (4.5,-10) {$\scriptsize\gyoung(<2>,<3>)$};

\draw[->] (11) --  node[pos=0.5,above]{\tiny $1$} (12);
\draw[->] (12) --  node[pos=0.5,above]{\tiny $0$} (13);
\draw[->] (13) --  node[pos=0.5,above]{\tiny $1$} (14);

\draw[->] (12) --  node[pos=0.5,left]{\tiny $2$} (22);
\draw[->] (13) --  node[pos=0.5,left]{\tiny $2$} (23);
\draw[->] (14) --  node[pos=0.5,left]{\tiny $2$} (24);

\draw[->] (22) --  node[pos=0.5,above]{\tiny $0$} (23);

\draw[->] (23) --  node[pos=0.5,left]{\tiny $1$} (33);
\draw[->] (24) --  node[pos=0.5,left]{\tiny $2$} (34);

\draw[->] (33) --  node[pos=0.5,left]{\tiny $1$} (43);
\draw[->] (34) --  node[pos=0.5,left]{\tiny $1$} (44);

\draw[->] (43) --  node[pos=0.5,left]{\tiny $0$} (53);
\draw[->] (44) --  node[pos=0.5,left]{\tiny $0$} (54);

\draw[->] (43) --  node[pos=0.5,above]{\tiny $2$} (44);

\draw[->] (53) --  node[pos=0.5,above]{\tiny $2$} (54);

\draw[->] (54) --  node[pos=0.5,left]{\tiny $1$} (64);
\end{tikzpicture}\end{center}
\end{Exa}

\newcommand{\cF}{\mathcal{F}}

Let us now fix $n,m\in\Z_{\geq 1}$. In the spirit of \cite[Section 2.2]{GL20},
we define a ``combinatorial Fock space'' which will naturally 
be endowed with a type $C_n$ crystal structure.    

\begin{Def}\label{KN_fock}
The \textit{Kashiwara-Nakashima (KN) Fock space} is the set
$$\cF_{n,m}=
\left\{  
c_1\otimes \cdots\otimes c_m \mid
c_j \text{ is a column on } \sC_n \text{ for all } j=1,\ldots, m
\right\}.$$
\end{Def}
We will be interested in a particular subset of the KN Fock space.
Denote by $\cC_{m,2n}$ the set of compositions
$\mu'=(\mu'_1,\ldots, \mu'_m)$ such that
$\max\left \{  \mu'_j \, ;\, 1\leq j\leq m \right\} \leq 2n$.
For $\mu'\in\cC_{m,2n}$, consider the set 
$$
 B_{\mu'}=\left\{ 
c_1\otimes \cdots \otimes c_m \in\cF_{n,m}
\mid |c_j|=\mu_j' \text{ for all } j=1,\ldots, m
\right\}.$$
By choosing to read the $m$ columns of any vertex of $B_{\mu'}$ \emph{first from left to right and next from top bottom}, we get an embedding of crystals
$ B_{\mu'}
\hookrightarrow
 B(\om_{n-1})^{\otimes |\mu|},$
and $B_{\mu'}$ realises the crystal of the representation 
$${\sf \Lambda}^{\mu'}_{2n,m}={\sf \Lambda}^{\mu'_1}(\C^{2n})\otimes \cdots \otimes {\sf \Lambda}^{\mu'_m}(\C^{2n})$$ of $\mathfrak{sp}_{2n}(\C)$ (which we have encountered in \Cref{Sec_HoweDuality}).
Moreover, we have the decomposition as direct sum of crystals
$$\cF_{n,m}=\bigoplus_{\mu'\in\cC_{m,2n}} B_{\mu'}.$$
Note that in the basis $(\eps_{\overline{n}},\ldots, \eps_{\overline{1}})$, we have
$\wt(b)=(a_n,\ldots, a_1)$ where, for all $i=1,\ldots, n$,
$$a_i=\# \text{ entries } \overline{i} \text{ in } b  - \# \text{ entries } i  \text{ in } b.$$
Inside $ B_{\mu'}$, we consider the following two subsets
$$ B_{\mu'}^\hw = \left\{ b\in  B_{\mu'} \mid  e_i(b) = 0 \text{ for all } i=0,\ldots, n-1 \right\}
\text{ \quad and \quad} B_{\mu',\la}^\hw = \left\{ b\in  B_{\mu'}^\hw \mid \wt(b)=\la  \right\}$$
for any $\la\in\cP_{n,m}$.
In other terms, $ B_{\mu'}^\hw$ is the set of highest weight vertices in $B_{\mu'}$.

\medskip

Let us recall briefly how to check that an element
$b\in B_{\mu'}$ is in $ B_{\mu'}^\hw$.
Let $\w$ be the word obtained by reading $b$ (as explained above). 
Define similarly the \textit{weight} of a word on $\sC_n$ to be the $n$-tuple whose
$i$-th coordinate is the difference between the number of $\overline{i}$'s
and the number of $i$'s.
Then $b$ is a highest weight vertex if and only if the weight of each prefix of $\w$ is 
a partition.

\begin{Exa}\label{Exa_hwv}
Let $n=4$, $m=3$, $\mu'=(2,3,1)$ and
$b=
\scriptsize\gyoung(<\overline{4}>,<\overline{3}>)
\otimes
\scriptsize\gyoung(<\overline{2}>,<\overline{1}>,1)
\otimes
\scriptsize\gyoung(<\overline{4}>)
\in  B_{\mu'}$, 
so that $\w=\bar{4}\bar{3}\bar{2}\bar{1}1\bar{4}$.
The prefixes of $\w$ are 
$
\bar{4},
\bar{4}\bar{3},
\bar{4}\bar{3}\bar{2},
\bar{4}\bar{3}\bar{2}\bar{1},
\bar{4}\bar{3}\bar{2}\bar{1}1,
\bar{4}\bar{3}\bar{2}\bar{1}1\bar{4}
$
with respective weights
$
(1,0,0,0),
(1,1,0,0),
(1,1,1,0),
(1,1,1,1),
(1,1,1,0),
(2,1,1,0),
$
which are all partitions, therefore $b$ is a highest weight vertex. 
More precisely, $b\in B_{(2,3,1),(2,1,1,0)}^\hw$.
\end{Exa}

\medskip

Consider now the alternative alphabet
$\sC_m^\ast = \left\{ 1<\overline{1}<\cdots <m<\overline{m}\right\}$.
Similarly to \Cref{KN_fock}, we can consider columns on $\sC^\ast_m$
and construct another combinatorial Fock space.

\newcommand{\dcF}{\dot{\cF}}

\begin{Def}\label{King_fock}
The \textit{King Fock space} is the set
$$\dcF_{m,n}=
\left\{  
d_1\otimes \cdots\otimes d_n \mid
d_i \text{ is a column on } \sC^\ast_m \text{ for all } i=1,\ldots, n
\right\}.$$
\end{Def}

\begin{Rem}\label{rem_King_fock}
Unlike $\cF_{n,m}$, there is no simple $C_m$-crystal structure on $\dcF_{m,n}$.
This will be discussed in more detail in \Cref{Sec_bicrystals}.
\end{Rem}

We are ready to define a duality
$$
\begin{array}{ccc}
\cF_{n,m} &\lra & \dcF_{m,n}\\
b & \longmapsto & b^\ast.
\end{array}
$$
Let $\mu'$ be a composition as before and $b=c_1\otimes  \cdots \otimes c_m \in  B_{\mu'}$.
For each $j=1,\ldots, m$,
let 
$$
\begin{array}{l}
\tc_j = \{ 1\leq x \leq n \mid \overline{x}\notin c_j\} \qu{and}
\\
\tc_{\overline{j}} = \{ 1\leq x \leq n \mid x\in c_j\} = \{1,\ldots, n\}\cap c_j,
\end{array}
$$
and set
$$
\tb=\tc_{1}\otimes \tc_{\overline{1}}\otimes\cdots\otimes \tc_{m}\otimes\tc_{\overline{m}}.
$$
Then $\tb$ is a tensor product of $2m$ columns of type $A_{n-1}$.
We can now apply the duality $\ast$ of \cite{GL20} to the element $\tb$ to get 
an element $b^\ast\in\dcF_{m,n}$.
More precisely, we set
$$b^\ast= d_1\otimes\cdots\otimes d_n \quad \text{ where } d_i=\left\{ x\in \sC_m^\ast    \mid i\in \tc_x\right\} \text{ for all } i=1,\ldots, n.$$

\begin{Exa}
Let $n=5$ and 
$b=c_1\otimes c_2 =
\scriptsize\gyoung(<\overline{3}>,<\overline{2}>,4,5)
\otimes
\scriptsize\gyoung(<\overline{5}>,<\overline{2}>,<\overline{1}>,1,2,4)$.
Then we have
$\tc_1 = \scriptsize\gyoung(1,4,5)$,
$\tc_{\overline{1}} = \scriptsize\gyoung(4,5)$,
$\tc_2 = \scriptsize\gyoung(3,4)$ and
$\tc_{\overline{2}} = \scriptsize\gyoung(1,2,4)$.
We obtain the following product of columns of type $A_4$:
$$\tb =  \tc_1\otimes \tc_{\overline{1}} \otimes \tc_2\otimes \tc_{\overline{2}} 
= \scriptsize\gyoung(1,4,5) \otimes \gyoung(4,5)\otimes\gyoung(3,4) \otimes \gyoung(1,2,4).$$
Finally, we find 
 $$b^\ast
= \scriptsize\gyoung(1,<\overline{2}>) \otimes \gyoung(<\overline{2}>)\otimes\gyoung(2) \otimes \gyoung(1,<\overline{1}>,2,<\overline{2}>)\otimes \gyoung(1,<\overline{1}>).$$
\end{Exa}

In particular, \cite[Proposition 2.17]{GL20} ensures that if $b\in  B_{\mu'}^\hw$,
then $b^\ast$ is a semistandard tableau (on the alphabet $\sC_m^\ast$),
where we identify a tableau with the tensor product of its columns (from left to right with our convention). 
In fact, we have more. In order to state the following theorem, recall that 
a tableau with entries in $\sC_m^\ast$ is called a \textit{King tableau}
if
\begin{enumerate}
\item it is semistandard, and 
\item if each entry in row $j$ is greater than or equal to $j$ for every index $j$.
\end{enumerate}
The \textit{weight} of a tableau $t$ with entries in $\sC_m^\ast$ is the sequence $\wt(t)=(a_m,\ldots, a_1)$
where, for all $j=1,\ldots, m$, 
$$a_j=\# \text{ entries } j \text{ in } b  - \# \text{ entries } \overline{j}  \text{ in } b.$$

In order to state the main result of this section, recall the involutive map $\beta'\mapsto \widehat{\beta}$
defined in \Cref{Sec_HoweDuality}.
Moreover, let us denote by $\cK_{\hla,\hmu}$ the set of King tableaux of shape $\hla$ and weight $\hmu$. 

\begin{Th}[Combinatorial Howe duality]\label{comb_duality_C}
We have a bijection
$$
\begin{array}{ccc}
 B_{\mu',\la}^\hw & \lra & \cK_{\hla,\hmu}
\\
b&\longmapsto & b^\ast.
\end{array}
$$
\end{Th}

Let us justify the name of the previous theorem.
By general crystal theory, the cardinality of  $ B_{\mu',\la}^\hw$
equals the multiplicity of $V(\la)$ in 
${\sf \Lambda}^{\mu'}_{2n,m}$.
Moreover, it is known that the cardinality of $\cK_{\hla,\hmu}$ equals the weight multiplicity~$K^{C_m}_{\wh{\la},\wh{\mu}}$,
see \cite{K76}.
Therefore, the bijection of \Cref{comb_duality_C} permits to recover the Howe duality
(\Cref{Thm_Howe_duality}). 
In fact, we directly obtain the more general version where 
the heights of the columns in the tensor product are not necessarily decreasing 
(that is, using compositions instead of partitions).

\begin{Exa}
Let us go back to \Cref{Exa_hwv}, where we had $b\in B_{\mu',\la}^\hw$
with $\mu'=(2,3,1)$ and $\la=(2,1,1,0)$.
We compute
$\tb=
\scriptsize \gyoung(1,2) \otimes \emptyset
\otimes
\scriptsize \gyoung(3,4) \otimes \gyoung(1)
\otimes
\scriptsize \gyoung(1,2,3) \otimes \emptyset
$,
which yields
$b^\ast=
\scriptsize \gyoung(1,<\overline{2}>,3)
\otimes
\scriptsize \gyoung(1,3)
\otimes
\scriptsize \gyoung(2,3)
\otimes
\scriptsize \gyoung(3)
$, represented by the tableau
$$\scriptsize
\gyoung(1123,<\overline{2}>33,3),$$
which is a King tableau of shape $(4,3,1)=\hla$ and weight $(3,1,2)=\hmu$.
\end{Exa}

\begin{proof}
Let us start by proving that the duality $\ast$ intertwines shape and weight as claimed.
Write $\wt(b^\ast)=\overline{\mu}=(\overline{\mu}_m,\ldots, \overline{\mu}_1)$.
From the definition of $\ast$, we see that, for all $j=1,\ldots, m$,
\begin{align*}
\overline{\mu}_j 
& = |\tc_{\overline{j}}|-|\tc_j|
\\
& = \# \text{ unbarred entries in } c_j - (n- \# \text{ barred entries in } c_j)
\\
& = n-|c_j|
\\
& = n-\mu_j'.
\end{align*}
Therefore, we have $\overline{\mu}=\hmu$.
Similarly, if we write $b^\ast=d_1\otimes\cdots\otimes d_n$ and 
$\sh(b^\ast)=\overline{\la}=(\overline{\la}_1,\ldots, \overline{\la}_m)$,
we have, for all $i=1,\ldots, n$,
\begin{align*}
\overline{\la}_i' 
& = |d_i|
\\
&= \sum_{j=1}^m \left( \left( 1-\indic_{c_j}(\overline{i}) \right) + \indic_{c_j}(i) \right)
\\
&= m - \sum_{j=1}^m \left( \indic_{c_j}(\overline{i}) - \indic_{c_j}(i) \right)
\\
& = m -\left( \# \text{ entries } \overline{i} \text{ in } b - \# \text{ entries } i \text{ in } b \right)
\\
& = m -\la_i,
\end{align*}
and taking the transpose yields $\overline{\la}=\hla$.

\medskip

In particular, if $b\in  B_{\mu'}^\hw$ then $\hla$ is a partition.
In fact, $b\in  B_{\mu'}^\hw$ is a highest weight vertex for the parabolic $A_{n-1}$-action 
if and only if $b^\ast$ is semistandard on $\sC_n^\ast$ by \cite[Proposition 2.17]{GL20}.
The only thing that remains to be proved is that
$ e_0(b)=0$ if and only if $b^\ast$ satisfies the Condition (2) defining King tableaux.
We prove this by induction on $|\mu|$.

\medskip

If $|\mu|=1$ then 
the only element of $ B_{\mu'}^\hw$ is
$b=\scriptsize\gyoung(<\overline{n}>)$. 
We have $b^\ast=\scriptsize\gyoung(<1>)\otimes\cdots\otimes \gyoung(<1>)\otimes \emptyset$, 
which is the only King tableau of weight $\hmu=(n-1)$.

\medskip

Fix $r\in\Z_{\geq 1}$ and assume that the claim holds for all $a\in  B_\nu$ with $|\nu|=r$.
Let $b\in  B_{\mu'}$ with $|\mu|=r+1$. Let $x$ be the bottommost entry of the rightmost column of $b$, and
let $a$ be the element obtained by deleting $x$ from $b$.
Assume first that $a$ and $b$ have the same number of non-trivial columns, say $k$.
If $x=j\geq 1$, then by definition, $b^\ast$ is obtained by adding a $\overline{k}$ 
in the $j$-th column of $a^\ast$.
By the induction hypothesis applied to $a=a_1\otimes \cdots\otimes a_n$,
one sees that Condition (2) holds for $b^\ast$ unless $j=1$ and $a_1=\{1,2,\ldots, k\}$.
This is equivalent to saying that the contribution in $\eps_{\overline{1}}$ in $\wt(b)$ is negative.
If $x=\overline{j}\leq \overline{1}$, then $b^\ast$ is obtained by adding a $k$ 
in the $j$-th column of $a^\ast$. One sees that Condition (2) holds for $b^\ast$ if and only if it holds for $a^\ast$,
and that $ e_0(b)=0$ if and only if $ e_0(a)=0$, and we conclude using the induction hypothesis applied to $a$.

\medskip

Finally, if $a$ has $k$ non-trivial columns and $b$ has $k+1$ non-trivial columns,
one uses similar arguments to validate the induction step in this case too.
\end{proof}

\begin{Rem} We conclude this section by mentioning two related recent results.
\begin{enumerate} 
\item In \cite[Theorem 2.7]{Lee19}, Lee constructed a weight-preserving bijection between certain semistandard oscillating tableaux on the one hand and certain King tableaux
on the other hand.
In fact, the highest weight vertices of \Cref{comb_duality_C}
correspond to certain semistandard oscillating tableaux (as illustrated in 
\Cref{Exa_hwv}).
Therefore, \Cref{Thm_Howe_duality} gives a simple proof of \cite[Theorem 2.7]{Lee19}
just based on the combinatorics of the columns of type $C_n$.
\item In \cite{HK20}, Heo and Kwon studied the type $C$ Howe duality
by using a symplectic version of the RSK algorithm.
Their results also involve King tableaux and
other combinatorial objects (the spinor model).
In our approach, we expect \Cref{comb_duality_C} 
to be related to a symplectic RSK correspondence
arising from a bicrystal structure on tensor products
of admissible columns in the spirit of \cite[Theorem 2.25]{GL20},
see also \Cref{Rem_bicrystal_2}(1) of the upcoming section.
\end{enumerate}
\end{Rem}

%%%%%%%%%%%%%%%%%%%%%%%%%%%%%%%%%%%%%%%%%%%%%%%%%%
%%%%%%%%%%%%%%%%%%%%%%%%%%%%%%%%%%%%%%%%%%%%%%%%%%
%%%%%%%%%%%%%%%%%%%%%%%%%%%%%%%%%%%%%%%%%%%%%%%%%%
%%%%%%%%%%%%%%%%%%%%%%%%%%%%%%%%%%%%%%%%%%%%%%%%%%
%%%%%%%%%%%%%%%%%%%%%%%%%%%%%%%%%%%%%%%%%%%%%%%%%%
%%%%%%%%%%%%%%%%%%%%%%%%%%%%%%%%%%%%%%%%%%%%%%%%%%

\section{Bicrystals and charge}
\label{Sec_bicrystals}

As mentioned in \Cref{rem_King_fock}, $\dcF_{m,n}$
does not come with a natural type $C_m$ crystal structure.
In fact, even on the subset of King tableaux, finding such a crystal structure is 
a challenging problem, see \Cref{Rem_bicrystal_2}.
However, there is a natural type $A_{2m-1}$ crystal structure
on $\dcF_{m,n}$ 
induced from the crystal of the vector representation below
\begin{center}
\begin{tikzpicture}
\node (a) at (-4.5,0) {$\scriptsize\gyoung(<1>)$};
\node (b) at (-3,0) {$\scriptsize\gyoung(<\overline{1}>)$};
\node (c) at (-1.5,0) {$\scriptsize\gyoung(<2>)$};
\node (d) at (0,0) {$\scriptsize\gyoung(<\overline{2}>)$};
\node (e) at (1.5,0) {$\cdots$};
\node (f) at (3,0) {$\scriptsize\gyoung(<m>)$};
\node (g) at (4.5,0) {$\scriptsize\gyoung(<\overline{m}>)$\ .};

\draw[->] (a) --  node[pos=0.5,above]{\tiny $1$} (b);
\draw[->] (b) --  node[pos=0.5,above]{\tiny $\overline{1}$} (c);
\draw[->] (c) --  node[pos=0.5,above]{\tiny $2$} (d);
\draw[->] (d) --  node[pos=0.5,above]{\tiny  $\overline{2}$} (e);
\draw[->] (e) --  node[pos=0.5,above]{\tiny $\overline{m-1}$} (f);
\draw[->] (f) --  node[pos=0.5,above]{\tiny $m$} (g);
\end{tikzpicture}
\end{center}

To compute the $A_{2m-1}$-crystal structure on $\dcF_{m,n}$,
we use a different reading than the one used in \Cref{Sec_Howe_by_crystals}.
More precisely, we choose this time to read the columns of $b\in\dcF_{m,n}$ 
\emph{first from right to left and next from top to bottom}.
The action of the crystal operators, denoted by $\dot{f}_{j}$, $1\leq j\leq m$,
and $\dot{f}_{\overline{j}}$, $1\leq j\leq m-1$
(resp. $\de_j$ and $\de_{\overline{j}}$)
on the resulting word 
is computed by using the same bracketing procedure as in \Cref{Sec_Howe_by_crystals}
(illustrated in \Cref{Ex_crystalrule1}),
where we encode this time $j$ by $+$ and $\overline{j}$ by $-$ for $\df_j,\de_j$,
and $\overline{j}$ by $+$ and $j+1$ by $-$ for $\df_{\overline{j}}, \de_{\overline{j}}$
(and we ignore the other letters).
This is illustrated in the example below.

\begin{Exa}\label{Exa_Acrystal}
Take $n=3,m=2$ and 
$b^\ast=
\scriptsize \gyoung(<\overline{1}>,2)
\otimes
\scriptsize \gyoung(1,<\overline{1}>,<\overline{2}>)
\otimes
\scriptsize \gyoung(1,2).
$
Let us detail the computation of
$\df_{\overline{1}} b^\ast$.
Reading $b^\ast$ yields the word 
$\w_{\overline{1}} = 121\bar{1}\bar{2}\bar{1}2$.
Looking only at $\overline{1}$ and $2$, we get $\w_{\overline{1}}=-++-$
and the bracketing yields $\w_{\overline{1}}=-+(+-)$.
Thus we get $\df_{\overline{1}} b^\ast= 
\scriptsize \gyoung(<\overline{1}>,2)
\otimes
\scriptsize \gyoung(1,2,<\overline{2}>)
\otimes
\scriptsize \gyoung(1,2)
$.
Similarly, one checks that 
$\de_{\overline{1}} b^\ast= 
\scriptsize \gyoung(<\overline{1}>,2)
\otimes
\scriptsize \gyoung(1,<\overline{1}>,<\overline{2}>)
\otimes
\scriptsize \gyoung(1,<\overline{1}>)
$\, ,
$\df_{2} b^\ast= 
\scriptsize \gyoung(<\overline{1}>,<\overline{2}>)
\otimes
\scriptsize \gyoung(1,<\overline{1}>,<\overline{2}>)
\otimes
\scriptsize \gyoung(1,2)
$
and that
$\de_1 b^\ast=\de_2 b^\ast=\df_1 b^\ast = 0$.

\end{Exa}

In the case $m=1$, the duality $\ast$ intertwines the following two important properties.

\begin{Prop}\label{admissible_dual} Let $c$ be a column on $\sC_n$.
Then $c$ is $n$-admissible if and only if $c^\ast$ is a highest weight vertex in the $A_1$-crystal.
\end{Prop}

\begin{proof}
By definition of the $A_1$-crystal structure on $\dcF_{m,n}$,
the element $c^\ast=c^\ast_1\otimes \cdots \otimes c^\ast_n$ is a highest weight vertex
if and only if for all $i=1,\ldots, n$, $c^\ast_i\otimes \cdots \otimes c^\ast_n$
has at least as many entries $1$ as $\overline{1}$.
Now, recall that $c$ is $n$-admissible if and only if $N_i(c)\leq n-i+1$ for all $i=1,\ldots, n$ (\Cref{Def_admissible}).
We can compute $N_i(c)$ on $c^\ast$ by using the definition of the duality: we get
$$N_i(c)= 
\left( \# \text{ entries $\overline{1}$ in $c_i^\ast\otimes\cdots\otimes c^\ast_n$ } \right)
+(n-i+1) - 
\left( \# \text{ entries $1$ in $c_i^\ast\otimes\cdots\otimes c^\ast_n$ } \right).$$
Since $N_i(c)\leq n-i+1$ we get the desired characterisation.
\end{proof}

\begin{Exa}
Let $n=6$ and 
$c=\scriptsize\gyoung(<\overline{3}>,<\overline{1}>,1,3,4,6)$, which is $n$-admissible.
Now, we can compute
$c^\ast = 
\scriptsize\gyoung(<\overline{1}>)
\otimes
\scriptsize\gyoung(1)
\otimes
\scriptsize\gyoung(<\overline{1}>)
\otimes
\scriptsize\gyoung(1,<\overline{1}>)
\otimes
\scriptsize\gyoung(1)
\otimes
\scriptsize\gyoung(1,<\overline{1}>)
$,
and we see that $c^\ast$ verifies the expected property.
\end{Exa}

\begin{Rem}
Similarly, there is a notion of coadmissibility for columns,
which is easily characterised on the dual.
More precisely, given a column $c$ of type $C_n$, set, for all $i=1,\ldots, n$,
$$ M_i(c)=\left| \left\{ x\in c \mid \overline{i}\leq  x \leq i  \right\}\right|.$$
Then $c$ is called \textit{$n$-coadmissible} if 
$M_i(c)\leq i$ for all $i=1,\ldots, n$.
As in \Cref{admissible_dual}, we can prove that
$c$ is $n$-coadmissible 
if and only if 
$c^\ast$ is a highest weight vertex with respect to the alternative $A_1$-crystal structure
where we choose to read the $n$ factors from left to right.
\end{Rem}

In general, we will now show some interesting
relationships between 
the $C_n$-crystal structure on $\cF_{n,m}$ and the $A_{2m-1}$-crystal 
structure on $\dcF_{m,n}$.

\medskip

For the next result, recall the type $C_n$ plactic relation called \textit{contraction} of a column
defined in \cite[Remark $\text{(ii)}$ p. 213]{Le05},
which is given by removing a certain pair $(\overline{k},k)$ appearing in the column.
It is easy to see that contraction has an inverse, which we call \textit{dilatation}.
Moreover, for $1\leq j\leq m$ define $\kappa_j: \cF_{n,m}\rightarrow \cF_{n,m}$
by setting $a= \kappa_j b$ if and only if $a^\ast = \de_j b^\ast$
(i.e. there is an arrow  $a^\ast \overset{\makebox[0pt]{\mbox{\tiny $j$}}}{\longrightarrow} b^\ast$).

\begin{Th}\label{contraction}
For all $1\leq j\leq m$, $\kappa_j$ is the contraction of the $j$-th column.
\end{Th}

\begin{proof}
Let $b=c_1\otimes\cdots\otimes c_m\in B_{\mu'}$ for some $\mu'\in\cC_{m,2n}$.
If $a^\ast = \de_j b^\ast$,
then $a^\ast$ is obtained from $b^\ast$ by changing a $\overline{j}$ into a $j$, say in column $k$.
That is, $a$ is obtained from $b$ by removing the pair $(\overline{k},k)$ in column $j$.
In fact, by construction, the contraction of the $j$-th column is the only $C_n$-crystal isomorphism
which removes such a pair.
Therefore, it suffices to prove that $\kappa_j$ is a $C_n$-crystal isomorphism
to deduce that $\kappa_j$ is the contraction of the $j$-th column.
So we will prove that
\begin{equation}\label{commutation}
\kappa_j e_i b = e_i \kappa_j b \quad \text{ for all } 0\leq i\leq n-1. 
\end{equation}
First of all, it is clear that (\ref{commutation}) holds in the following two cases:
\begin{itemize}
 \item  $e_i$ does not act on the $j$-th column of $b$, 
 \item  $i\notin \{k, k-1\}$. 
\end{itemize}
Let us look at the remaining cases, that is, $e_i$ acts on the $j$-th column of $b$ and
\begin{enumerate}
\item $i=k-1$. In this case, since we already know that $\overline{k}, k\in c_j$,
we are ensured that $e_i$ acts non trivially on $b$ if and only if $\overline{k-1}\in c_j$ and $k-1\notin c_j$.
This means that we have the configuration
$$c_j=\scriptsize\gyoung(|\sesqui\vdts,<\overline{k}>,<\overline{k\text{-}1}>,|\sesqui\vdts,<k>,|\sesqui\vdts).$$
One checks that applying $\kappa_j e_i$ amounts to
\begin{enumerate}
 \item changing $k$ into $k-1$, followed by
 \item removing the pair $(\overline{k-1},k-1)$
\end{enumerate}
On the other hand, applying $e_i \kappa_j$ amounts to
\begin{enumerate}
 \item removing the pair $(\overline{k},k)$, followed by
 \item changing $\overline{k-1}$ into $\overline{k}$
\end{enumerate}
Both of these procedures yield the same result, namely, we get the $j$-th column
$$c_j=\scriptsize\gyoung(|\sesqui\vdts,<\overline{k}>,|\sesqui\vdts)$$
(where $\overline{k-1}$ and $k$ have been deleted), and we have (\ref{commutation}) as expected.
\item $i=k$. Similarly, $e_i$ acts non trivially on $b$ if and only if $\overline{k+1}\notin c_j$ and $k+1\in c_j$,
and this means that we have the configuration
$$c_j=\scriptsize\gyoung(|\sesqui\vdts,<\overline{k}>,|\sesqui\vdts,<k>,<k\text{\tiny +}1>,|\sesqui\vdts).$$
In this case, one checks that (\ref{commutation}) holds and that the resulting column is
$$c_j=\scriptsize\gyoung(|\sesqui\vdts,<k>,|\sesqui\vdts)\ .$$
\end{enumerate}
\end{proof}

\begin{Exa}
We illustrate Case (1) of the previous proof by taking
$n=4$, $m=1$ and 
$b=c_1=\scriptsize\gyoung(<\overline{4}>,<\overline{3}>,<\overline{2}>,3)$.
We have $j=1$, we check that $k=3$ and
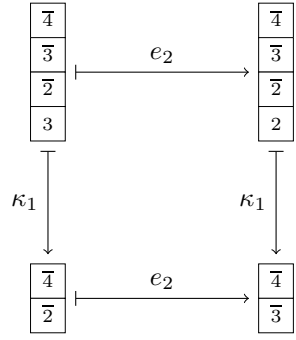
\begin{center}
\begin{tikzpicture}
\node (11) at (0,0) {$\scriptsize\gyoung(<\overline{4}>,<\overline{3}>,<\overline{2}>,3)$};
\node (12) at (3,0) {$\scriptsize\gyoung(<\overline{4}>,<\overline{3}>,<\overline{2}>,2)$};
\node (21) at (0,-3) {$\scriptsize\gyoung(<\overline{4}>,<\overline{2}>)$};
\node (22) at (3,-3) {$\scriptsize\gyoung(<\overline{4}>,<\overline{3}>)$};

\draw[|->] (11) --  node[pos=0.5,above]{$e_2$} (12);
\draw[|->] (11) --  node[pos=0.5,left]{$\kappa_1$} (21);
\draw[|->] (21) --  node[pos=0.5,above]{$e_2$} (22);
\draw[|->] (12) --  node[pos=0.5,left]{$\kappa_1$} (22);
\end{tikzpicture}
\end{center}
\end{Exa}

Therefore, we shall 
consider the $A_1\times\cdots\times A_1$-crystal structure ($m$ factors) on $\dcF_{m,n}$
induced by keeping only arrows of the form 
${\scriptsize\gyoung(<j>)} \overset{\makebox[0pt]{\mbox{\tiny $j$}}}{\longrightarrow} \scriptsize\gyoung(<\overline{j}>)$.
\Cref{contraction} directly implies the following corollary,
which can be rephrased by saying that the combinatorial Fock space is endowed with a $(C_m \times A_1^m)$-bicrystal structure.

\begin{Cor}\label{cor_bicrystal}
The  duality $\ast$ intertwines the $C_n$-crystal on $\cF_{n,m}$ and the $A_1\times\cdots\times A_1$-crystal on $\dcF_{m,n}$.
\end{Cor}

We complete \Cref{contraction} by giving
a characterisation of the dual $\kappa_{\overline{j}}$ of the Kashiwara operators
$e_{\overline{j}}$ (corresponding to the arrows
${\scriptsize\gyoung(<\overline{j}>)} \overset{\makebox[0pt]{\mbox{\tiny $\overline{j}$}}}{\longrightarrow} \scriptsize\gyoung(<j\text{+}1>)$) 
in terms of jeu de taquin operators on $\cF_{n,m}$.
More precisely, if $b=c_1\otimes  \cdots \otimes c_m\in B_{\mu'}$,
set for all $1\leq j\leq m$,
$$
\begin{array}{l}
\overline{c}_j = \{ \overline{n}\leq \overline{x} \leq \overline{1} \mid \overline{x}\in c_j\}
=\{\overline{n},\ldots, \overline{1}\}\cap c_j
\qu{and}
\\
\overline{c}_{\overline{j}} = \{ \overline{n}\leq \overline{x} \leq \overline{1} \mid x \notin c_j\},
\end{array}
$$
so that each $\overline{c}_j$ is the complement of the column $\tc_j$ defined in \Cref{Sec_Howe_by_crystals}.
Also, we set
$$
\overline{b}=\overline{c}_{1}\otimes \overline{c}_{\overline{1}}\otimes\cdots\otimes \overline{c}_{m}\otimes\overline{c}_{\overline{m}}.
$$
On elements of the form $\overline{b}$,
consider for $1\leq j\leq m-1$ the jeu de taquin operator $\mJ_{\overline{j}}$ acting on columns
$\overline{j}$ and $j+1$ as illustrated in \Cref{Exa_jdt} below.

\begin{Exa}\label{Exa_jdt}
Let $n=5, m=2$ and 
${b}=
{
\scriptsize
\gyoung(<\overline{3}>,1,5)
\otimes
\gyoung(<\overline{5}>,<\overline{1}>,2,4,5)
}
$
so that
$\overline{b}
=\overline{c}_{1}\otimes \overline{c}_{\overline{1}}\otimes\overline{c}_{2}\otimes \overline{c}_{\overline{2}}
=
{
\scriptsize
\gyoung(<\overline{3}>)
\otimes
\gyoung(<\overline{4}>,<\overline{3}>,<\overline{2}>)
\otimes
\gyoung(<\overline{5}>,<\overline{1}>)
\otimes
\gyoung(<\overline{3}>,<\overline{1}>)
}.
$
Let us perform one jeu de taquin operation between 
$\overline{c}_{\overline{1}}= \scriptsize \gyoung(<\overline{4}>,<\overline{3}>,<\overline{2}>)$ 
and $\overline{c}_{2}=\scriptsize
\gyoung(<\overline{5}>,<\overline{1}>)$.
We want to slide a box from $\overline{c}_{\overline{1}}$ 
to $\overline{c}_2$, so we reverse the order of the two columns, which yields
the following jeu de taquin operation
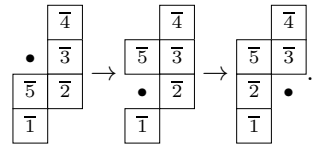
\Yvcentermath1
$$
{\scriptsize\gyoung(:~<\overline{4}>,:\bullet<\overline{3}>,<\overline{5}><\overline{2}>,<\overline{1}>)}
\to 
{\scriptsize\gyoung(:~<\overline{4}>,<\overline{5}><\overline{3}>,:\bullet<\overline{2}>,<\overline{1}>)}
\to 
{\scriptsize\gyoung(:~<\overline{4}>,<\overline{5}><\overline{3}>,<\overline{2}>:\bullet,<\overline{1}>)}.
$$
Finally, we get
$$\mJ_{\overline{j}} (\overline{b})=
{
\scriptsize
\gyoung(<\overline{3}>)
\otimes
\gyoung(<\overline{4}>,<\overline{3}>)
\otimes
\gyoung(<\overline{5}>,<\overline{2}>,<\overline{1}>)
\otimes
\gyoung(<\overline{3}>,<\overline{1}>)
}.
$$
\end{Exa}

We extend the definition of $\mJ_{\overline{j}}$ to $\cF_{n,m}$ by setting
$\overline{\mJ_{\overline{j}}(b)}=\mJ_{\overline{j}}(\overline{b})$.
We are ready to state the desired result, which is a symplectic analogue of \cite[Theorem 2.9]{GL20}.

\begin{Th}\label{jdt}
For all $1\leq j\leq m-1$, we have 
$\kappa_{\overline{j}}=\mJ_{\overline{j}}$.
\end{Th}

\begin{proof}
We use a similar argument to the proof of \Cref{contraction}.
Namely, we first notice that
if $a^\ast =  e_{\overline{j}} b^\ast$, then~$a^\ast$ is obtained from $b^\ast$ by changing a $j+1$ into a $\overline{j}$, say in column $k$.
By definition of $\ast$, this means that.
$\overline{a}$ is obtained from $\overline{b}$ by
moving an entry $\overline{k}$ from column $\overline{j}$ to column $j+1$.
Then we use the fact that $\mJ_{\overline{j}}$ is 
the only map that verifies this property 
and that commutes with the $C_n$-crystal operators $e_1,\ldots, e_{n-1}$,
and show by case analysis that this also holds for $\kappa_{\overline{j}}$.
\end{proof}

\begin{Rem}\label{Rem_bicrystal}
Unlike the operators $\kappa_j$, we do not have that $\kappa_{\overline{j}}$ commute with the type $C_n$
Kashiwara operator~$e_0$.
For instance, if $n=1, m=2$, and $b=\scriptsize\emptyset\otimes\gyoung(1)$,
one checks that $\kappa_{\overline{1}} e_0 b = 0$ but 
$e_0 \kappa_{\overline{1}} b = \scriptsize\gyoung(<\overline{1}>)\otimes\gyoung(<\overline{1}>,1)$.
Therefore, we do not get an analogue of \Cref{cor_bicrystal}.
\end{Rem}

\begin{Exa}
Let $n=5, m=2, j=1$ and
$$
b^\ast = 
{
\scriptsize
\gyoung(<1>,<\overline{1}>,<\overline{2}>)
\otimes
\gyoung(<2>)
\otimes
\gyoung(1,2)
\otimes
\gyoung(<\overline{1}>,2)
\otimes
\gyoung(1,<\overline{1}>)
}
\quad \text{ so that } \quad 
b= 
{
\scriptsize
\gyoung(<\overline{4}>,<\overline{2}>,1,4,5)
\otimes
\gyoung(<\overline{5}>,<\overline{1}>,1)
}.
$$
We get
$$
e_{\overline{1}} b^\ast =
{
\scriptsize
\gyoung(<1>,<\overline{1}>,<\overline{2}>)
\otimes
\gyoung(<\overline{1}>)
\otimes
\gyoung(1,2)
\otimes
\gyoung(<\overline{1}>,2)
\otimes
\gyoung(1,<\overline{1}>)
},
\quad
\text{ so $k=2$ and }
\kappa_{\overline{1}} b =
{
\scriptsize
\gyoung(<\overline{4}>,<\overline{2}>,1,2,4,5)
\otimes
\gyoung(<\overline{5}>,<\overline{2}>,<\overline{1}>,1)
}.
$$
On the other hand, we have
$$\overline{b}=
{
\scriptsize
\gyoung(<\overline{4}>,<\overline{2}>)
\otimes
\gyoung(<\overline{3}>,<\overline{2}>)
\otimes
\gyoung(<\overline{5}>,<\overline{1}>)
\otimes
\gyoung(<\overline{5}>,<\overline{4}>,<\overline{3}>,<\overline{2}>)
}.$$
The jeu de taquin operation corresponding to $\mJ_{\overline{1}}$ is
$$
{\scriptsize\gyoung(:\bullet<\overline{3}>,<\overline{5}><\overline{2}>,<\overline{1}>)}
\to 
{\scriptsize\gyoung(<\overline{5}><\overline{3}>,:\bullet<\overline{2}>,<\overline{1}>)}
\to 
{\scriptsize\gyoung(<\overline{5}><\overline{3}>,<\overline{2}>:\bullet,<\overline{1}>)},
$$
and we check that this yields $\mJ_{\overline{1}} (b) = \kappa_{\overline{1}} (b).$
\end{Exa}

At this point, we make a small digression and explain briefly
the relationship with the charge statistic 
defined in \cite[Theorem 6.13]{LL18} to compute $q$-weight zero multiplicities.

\medskip

Let $b\in B_{\mu'}^\hw$ such that $b^\ast$ is a King tableau (\Cref{comb_duality_C}) 
of weight $\widehat{\mu}=0$. In particular, we have $\mu'=(m^n)$, that is all the columns in $\mu$ have height $n$.
Denote $b^\ast_{\mathrm{low}}$ 
the lowest weight vertex corresponding to $b^\ast$
in the $A_1\times \cdots \times A_1$-crystal. In other terms,
$b^\ast_{\mathrm{low}}$ is obtained from $b^\ast$ by applying all possible
$\df_j$ to $b^\ast$ recursively (they commute).

\newcommand{\ch}{\mathrm{ch}}

\begin{Def}
The \textit{charge} of $b^\ast$ is the nonnegative integer
$$\ch(b^\ast) =
\sum_{j=1}^m (2(m-j)+1)\frac{\eps_j(b^\ast_{\mathrm{low}})}{2}
+
\sum_{j=1}^{m-1} 2(m-j)\left\lceil\frac{\eps_{\overline{j}}(b^\ast_{\mathrm{low}})}{2}\right\rceil
$$
\end{Def}

We can in fact construct a statistic $D$ directly on $B_{\mu'}$
such that $$D(b)=\ch(b^\ast).$$

\medskip

To do this, starting from $b=c_1\otimes\cdots\otimes c_m\in B_{\mu'}$.
let $b_\mathrm{dil}=d_1\otimes \cdots\otimes d_m$
where, for all $j=1,\ldots, m$, $d_j$ is the type $C_n$ 
column obtained by dilating $c_j$ recursively as much as possible.
Then $d_j$ can be contracted a certain number of times, say $\delta_j$, until it becomes admissible.
At this point, we have by \Cref{contraction}
$$\left(b_\mathrm{dil}\right)^\ast = b^\ast_\mathrm{low},$$
and 
$$\eps_j(b^\ast_\mathrm{low})=\delta_j.$$
In fact, we can easily compute $\delta_j$. 
Let $h$ denote the height of the admissible column corresponding to $c_j$ (i.e. obtained from $c_j$ by
applying recursively as many contractions as possible).

\begin{Prop}
We have then $\delta_j=n-h$. Moreover, $\delta_j$ is even.
\end{Prop}

\begin{proof}
Note that since $|c_j|=n$, the difference of the sizes $n-h$ is even (since each contraction deletes $2$ entries).
Therefore, it suffices to show that $\delta_j=n-h$.
Denote $c^\mathrm{adm}_j$ the admissible columns corresponding to $c_j$, so that $h=|c^\mathrm{adm}_j|$.
Since contraction is a $C_n$-crystal isomorphism,
we can compute $\delta_j$ by considering the highest weight column associated to $c^\mathrm{adm}_j$ in the $C_n$-crystal.
Obviously, this column also has height $h$ and is admissible (since $c^\mathrm{adm}_j$ is admissible), so it must be the column
$\{ \overline{n}, \overline{n-1},\ldots, \overline{n-h+1}\}$.
It is straightforward to see that it can be dilated at most $n-h$ times, resulting in the column
$\{ \overline{n}, \overline{n-1},\ldots, \overline{n-h+1}, \overline{n-h},\ldots, \overline{1},1,\ldots, n-h\}$
of height $2n-h$ (which is the highest weight column associated to $d_j$), that is,
$\delta_j=n-h$.
\end{proof}

It remains to express   $\eps_{\overline{j}} (b^\ast_\mathrm{low})$ directly on $b_\mathrm{dil}$.
By \Cref{jdt}, this equals the number $\gamma_j$ 
of successive possible application of $\mJ_{\overline{j}}$ to the pair 
$(\overline{d}_{\overline{j}},\overline{d}_{j+1})$.
In other terms, if the minimal skew tableau associated to $(\overline{d}_{\overline{j}},\overline{d}_{j+1})$ has skew shape $\nu/(1^\ell)$, we have $\gamma_j=\ell$.
To sum up, we set
$$D(b)=
\sum_{j=1}^m (2(m-j)+1)\frac{\delta_j}{2}
+
\sum_{j=1}^{m-1} 2(m-j)\left\lceil\frac{\gamma_j}{2}\right\rceil.
$$
This can be seen as an analogue of the energy statistic for type $C$,
naturally generalising \cite[Theorem 2.51]{GL20}.

\medskip

We end this section by three important remarks.

\begin{Rem}\label{Rem_bicrystal_2} \
\begin{enumerate}
\item  As mentioned in \Cref{Rem_bicrystal},
since the operators $\kappa_{\overline{j}}$ do not commute with $e_0$,
we cannot establish a $(C_n\times A_{2m-1})$-bicrystal structure on $ B_{\mu'}$.
In another direction, we expect to obtain a $(C_n\times C_m)$-bicrystal structure
by considering appropriate jeu de taquin operations and contraction on columns.
In this case, it would be interesting to compare the resulting bicrystal structure to that obtained by Lee in \cite{Lee19}.
\item Let us consider elements of $B_{\mu'}^\hw$ that
are products of admissible columns. Combining \Cref{comb_duality_C} and
\Cref{admissible_dual}, these are in duality with King tableaux of weight $\widehat{\mu}$
which are highest weight vertices in the $A_1^m$-crystal.
This can be seen as a combinatorial version the new duality which will be proved in \Cref{generalised_duality},
in the special case $r=m$, $\bm=(1,\ldots, 1)$, $\bX=(A,\ldots, A)$ and $\bmu=((1^{\mu'_1}),\ldots,(1^{\mu'_m}))$
(columns of height $\mu'_j$).
\item The energy function $D$ defined previously does not coincide with the intrinsic energy 
function on type~$C_n^{(1)}$ tensor products of column Kirillov-Reshetikhin crystals. 
This reflects the fact that the $q$-weight multiplicities do not coincide with the 
one-dimensional sums beyond type $A$. 
Nevertheless, this suggests that other interesting statistics could exist on these particular affine crystals. 
\end{enumerate}
\end{Rem}

%%%%%%%%%%%%%%%%%%%%%%%%%%%%%%%%%%%%%%%%%%%%%%%%%%
%%%%%%%%%%%%%%%%%%%%%%%%%%%%%%%%%%%%%%%%%%%%%%%%%%
%%%%%%%%%%%%%%%%%%%%%%%%%%%%%%%%%%%%%%%%%%%%%%%%%%
%%%%%%%%%%%%%%%%%%%%%%%%%%%%%%%%%%%%%%%%%%%%%%%%%%
%%%%%%%%%%%%%%%%%%%%%%%%%%%%%%%%%%%%%%%%%%%%%%%%%%
%%%%%%%%%%%%%%%%%%%%%%%%%%%%%%%%%%%%%%%%%%%%%%%%%%

\section{Branching coefficients and a new duality}

The aim of this section is to extend the results of Section \ref{Sec_HoweDuality} to the case where the
fundamental $\mathfrak{gl}_{2n}(\mathbb{C})$-modules appearing in the tensor
products  in Theorem \ref{Thm_Howe_duality} are replaced by tensor products of simple 
$\mathfrak{gl}_{2n}(\mathbb{C})$ or simple $\mathfrak{sp}_{2n}(\mathbb{C})$-modules. 
More precisely,  let $\bX$ be a sequence $(X_{1},\ldots,X_{r})$  of symbols in $\{A,C\}^{r}$ and let~$\bmu=(\mu^{(1)},\ldots,\mu^{(r)})$ be a sequence of partitions such that $\mu^{(j)}\in \cP_{n\times m_j}$.  Then one can associate to $(\bX,\bmu)$ the tensor product
$$V^{X_1}(\mu^{(1)})\otimes V^{X_2}(\mu^{(2)})\otimes \cdots \otimes V^{X_r}(\mu^{(r)})$$
where $V^{C}(\delta)$ ($\delta\in \cP_n$) denotes the irreducible $\mathfrak{sp}_{2n}(\C)$-module of highest weight $\delta$ and $V^{A}(\delta)$ denotes the restriction to $\mathfrak{sp}_{2n}(\C)$ of the irreducible
$\mathfrak{gl}_{2n}(\C)$-module of highest weight $\delta$. We show that the mutliplicity of the highest weight module of weight $\la$ in the tensor product above is a branching coefficient of the form $[V^{C_m}(\widehat{\lambda}):V^{\bX^\ast}_{\bm^\ast}(\widehat{\bmu})]$ where $m=\sum m_j$ and $V^{\bX^\ast}_{\bm^\ast}(\widehat{\bmu})$ is an irreducible highest weight module for a block diagonal subalgebra of~$\frak{sp}_{2m}(\C)$ that depends on $\bX$, $\bmu$ and $\bm$. 

\begin{Rem}
The weights of $\mathfrak{gl}_{2n}(\C)$ are in bijection with non-increasing sequences of integers in $\Z^n$. If two such sequences $\delta,\delta'$ differ by a multiple of $(1,\ldots,1)\in \Z^n$, then the corresponding Schur functions~$\fs_{\delta}$ and~$\fs_{\delta'}$ will differ by a power of $x_1\cdots x_{2n}$. Since the characters of  $V^{A}(\delta)$ and $V^{A}(\delta')$ are the specialisation of~$\fs_{\delta}$ and $\fs_{\delta'}$ at $(x_1,\ldots,x_n,x_1^{-1},\ldots,x_n^{-1})$, these two characters will be equal. As a consequence, it is enough in the tensor products defined above to restrict ourself to partitions.
 \end{Rem}

\medskip

Our first task in this section is to defined the module $V^{\bX^\ast}_{\bm^\ast}(\widehat{\bmu})$. We start by extending the definition of the map~$\wh{\ \cdot\ }$ to $r$-tuples of partitions. Let  $\bmu=(\mu^{(1)},\ldots,\mu^{(r)})$ be a sequence of partitions such that $\mu^{(j)}\in \cP_{n\times m_j}$. We set
$$\wh{\bmu} =  \big(\wh{\mu^{(r)}},\ldots,\wh{\mu^{(1)}}\big)$$
where for each $1\leq j\leq r$, the partition $\wh{\mu^{(j)}}$ is defined with respect to the pair  $(n\times m_{j})$ and lies in $\cP_{m_j\times n}$.
Note that the definition of this map depends on the pair $(n,\bm)$. 

\begin{Rem}
\begin{enumerate}
\item Let $\mu\in \cP_{n,m}$ and write $\mu'=(\mu'_1,\ldots,\mu'_m)$. Let $\bm = (m_1,\ldots,m_r)$ be such that $\sum m_j = m$. If we see $\mu$ as an $r$-tuple of partitions $\bmu = (\mu^{(1)},\ldots,\mu^{(r)})$ as follows
$$\big(
\underset{\mu^{(1)}}{\underbrace{\mu'_1,\ldots,\mu'_{m_1}}},
\underset{\mu^{(2)}}{\underbrace{\mu'_{m_1+1},\ldots,\mu'_{m_2}}},
\ldots,
\underset{\mu^{(r)}}{\underbrace{\mu'_{m-m_r+1},\ldots,\mu'_{m_r}}}\big)
$$
then the partition $\wh{\bmu}$ computed with respect to the pair $(n,\bm)$ is equal, as an element of $\Z^m$, to the partition $\wh{\mu}$ computed with respect to the pair $(n,m)$.
\item Let $\bmu = (\mu^{(1)},\ldots,\mu^{(r)})$ be such that $\mu^{(j)}\in \cP_{n,m_j}$ and let $\bmu'=({\mu^{(1)}}',\ldots,{\mu^{(r)}}')$.
Let $\bm = (m_1,\ldots,m_r)$ and $m = \sum m_j$. We have
$$\wh{\bmu}= \sI(\bmu') + n\cdot \omega_m$$
where $\omega_m = (1,\ldots,1)$ and  the equality is an equality in $\Z^m$. 
\end{enumerate}
\end{Rem}

\begin{Exa}
\label{widehat}
Consider the sequence of partitions 
$$\bmu = (\mu^{(1)},\mu^{(2)},\mu^{(3)}) = \bigg(\ 
\scalebox{.5}{\ydiagram{2,1,1}}\ ,\ 
\scalebox{.5}{\ydiagram{2}}\ ,\ 
\scalebox{.5}{\ydiagram{3,2}}\ \bigg)\in \cP_{3,2}\times \cP_{3,2}\times \cP_{3,3}.$$
Let $n=3$ and $\bm=(2,2,3)$.
We compute the image of $\bmu$ under the map $\wh{\ \cdot\ }$ 
associated to $(n,\bm)$ by first taking the complement of the partitions of $\bmu$ (in their respective rectangles), which yields
$$
\bigg(\ 
\scalebox{.5}{\ydiagram[*(white)]{2,1,1}*[*(green!50)]{2,2,2}}\ ,\ 
\scalebox{.5}{\ydiagram[*(white)]{2}*[*(green!50)]{2,2,2}}\ ,\ 
\scalebox{.5}{\ydiagram[*(white)]{3,2}*[*(green!50)]{3,3,3}}\ \bigg)\in \cP_{3,2}\times \cP_{3,2}\times \cP_{3,3},$$
and then by taking the conjugate of each green partition and reversing the order of the triple,
which yields
$$\wh{\bmu} = \bigg(\ 
\scalebox{.5}{\ydiagram{2,1,1}}\ ,\ 
\scalebox{.5}{\ydiagram{2,2}}\ ,\ 
\scalebox{.5}{\ydiagram{2}}\ \bigg)\in \cP_{3,3}\times \cP_{2,3}\times \cP_{2,3}.$$
Finally, note that $m=7$ and that 
\begin{align*}
\sI(\bmu') + 3\cdot \omega_7 &= \sI(3,1,1,1,2,2,1)  + 3\cdot \omega_7 \\
&= (-1,-2,-2,-1,-1,-1,-3)  + 3\cdot \omega_7 \\
& = (2,1,1,2,2,2,0)
\end{align*}
which is indeed equal to $\wh{\bmu}$ as an element of $\Z^7$.
\end{Exa}

Next we need to define the module  $V^{\bX^\ast}_{\bm^\ast}(\widehat{\bmu})$. To simplify the exposition, we will explain how to construct the module $V^{\bX}_{\bk}(\bnu)$ associated to 
\begin{itemize}
\item a sequence $\bk = (k_1,\ldots,k_r)$ of positive integers;
\item a sequence $\bX = (X_1,\ldots,X_r)$ of symbols in $\{A,C\}$;
\item a sequence $\bnu = (\nu^{(1)},\ldots,\nu^{(r)})$ of partitions such that $\nu^{(j)}\in\cP_{k_j}$.
\end{itemize}
We set $m = \sum k_i$ and we define the integers $K_0,\ldots,K_{r}$ by setting
$$K_0 = 0 \qu{ and } K_j = \sum_{i=1}^{j} k_i.$$
The  algebra $\mathfrak{g}_{(\bX,\bk)}$ is defined to be the subalgebra of $\mathfrak{sp}_{2k}$ associated to the root system
$$R_{(\bX,\bk)}=\bigsqcup_{j=1}^{r}R^{X_j}_{[K_{j-1}+1,K_{j}]}$$
where  we have set for any pair of integers $(a,b)$:
\begin{align*}
R^{A}_{[a,b]}  &  =\{\varepsilon_{i}-\varepsilon_{j}\mid a\leq i<j\leq
b\}\text{ and }\\
R_{[a,b]}^{C}  &  =\{\varepsilon_{i}-\varepsilon_{j}\mid a\leq i<j\leq
b\}\cup\{\varepsilon_{i}+\varepsilon_{j}\mid a\leq i<j\leq b\}\cup \{2\eps_i\mid a\leq i \leq b\}.
\end{align*}
In other words, we have
$$\mathfrak{g}_{(\bX,\bk)}=\mathfrak{g}_{1}\oplus\cdots\oplus\mathfrak{g}_{r}
\qu{where}
\mathfrak{g}_{j}=\left\{
\begin{array}
[c]{c}%
\mathfrak{gl}_{k_{j}}(\C)\text{ if }X_{j}=A\\
\mathfrak{sp}_{2k_{j}}(\C)\text{ if }X_{j}=C
\end{array}
\right.
$$
Let $P_{+}^{(\bX,\bk)}$ be the set of dominant weights
of $\mathfrak{g}_{(\bX,\bk)}$. A weight $\bnu$ in $P_{+}^{(\bX,\bk)}$ is a sequence $\bnu=(\nu^{(1)},\ldots,\nu^{(r)})$ of partitions such that $\nu^{(j)}\in \cP_{k_j}$. We define $V^{\bX}_{\bk}(\bnu)$ to be the $\mathfrak{g}_{(\bX,\bk)}$-module of highest weight $\bnu$.

\medskip

Let us now recall how to compute the branching
coefficient  $[V(\kappa):V^{\bX}_{\bk}(\bnu)]$ where $\kappa\in \cP_m$. To do so, we need
the partition function $\mathcal{P}_{(\bX,\bk)}$ defined
by the expansion%
$$
 \prod_{\al\in R^+\backslash R_{(\bX,\bk)}} \dfrac{1}{1-x^\al}
= \sum_{\beta\in \Z^m} \cP_{(\bX,\bk)}(\beta)\sx^\beta.
$$
The following proposition is a consequence of the Weyl character formula. Recall that $W_m$ is the Weyl group of type $C_m$ and that $\rho_m$ is the half sum of positive roots, that is $\rho_m = (m,\dots,1)$.

\begin{Prop}[{\cite[Theorem 8.2.1]{GW99}}]
\label{Prop_Branching}With the previous notation, we have
\[
\left[V(\kappa):V^{\bX}_{\bk}(\bnu)\right]=\sum_{w\in W_{m}}%
\varepsilon(w)\mathcal{P}_{(\bX,\bk)}(w(\kappa+\rho
_{m})-(\boldsymbol{\nu}+\rho_{m})).
\]

\end{Prop}

We are now ready to state the main theorem of this section. 
For any tuple $\ba = (a_1,\ldots,a_k)$ we write $\ba^\ast$ for the inverse tuple $(a_k,\ldots,a_1)$. 
\begin{Th}
\label{generalised_duality}
Let $\bmu=(\mu^{(1)},\ldots,\mu^{(r)})$ be a sequence of partitions such that $\mu^{(j)}\in \cP_{n\times m_j}$.  Let $\bX=(X_1,\ldots,X_r)$ be a sequence of symbols in $\{A,C\}$ and  $\bm = (m_1,\ldots,m_r)\in\Z_{\geq 1}^r$. For all $\la\in \cP_n$, we have
$$\left[ V^{X_{1}}(\mu^{(1)})\otimes\cdots\otimes V^{X_{r}} (\mu^{(r)}) : V(\la) \right] = 
\left[V(\widehat{\lambda}):V^{\bX^\ast}_{\bm^\ast}(\widehat{\bmu})\right].
$$
\end{Th}
The rest of this section is devoted to the proof of the theorem. From now on and until the end of the section, we fix $\bmu$, $\bX$ and $\bm$ as in the theorem. Further, we set $m = \sum m_i$ and we define the integers $M_0,\ldots,M_{r}$ by the relations
$$M_0 = 0 \qu{ and } M_j = \sum_{i=1}^{j} M_i.$$
For any $j=1,\ldots,r,$ and any
partition $\delta\in \cP_n$ write 
\begin{itemize}
\item $\fs_{\delta}^{C}$ for the character of the irreducible module $V(\delta)$ of highest weight $\delta$ in $\frak{sp}_{2n}(\C)$ 
\item $\fs_{\delta}^{A}$ for the function $\fs_{\delta}(x_{1},\ldots,x_{n},\frac{1}{x_{n}},\ldots,\frac{1}{x_{1}})$ where $\fs_\delta$ is the Schur function of type $A$ in $2n$ variables.
\end{itemize}
The character of $ V^{X_{1}}(\mu^{(1)})\otimes\cdots\otimes V^{X_{r}} (\mu^{(r)})$ is
\[
\fs_{\bmu}^{(\bX,\bm)}:=\prod_{j=1}^{r}%
\fs_{\mu^{(j)}}^{X_{j}}\in\text{$\mathrm{char}$}(\mathfrak{sp}_{2n}).
\]
Thus, if we define the coefficients $m_{\lambda,\bmu}^{(\bX,\bm)}$ by decomposing the above character in the basis of irreducible characters in type $C$
\[
\fs_{\bmu}^{(\bX,\bm)}=\sum_{\lambda
\in\mathcal{P}_{n}}m_{\lambda,\bmu}^{(\bX%
,\bm)}\fs_{\lambda}^{C_{n}}%
\]
then the theorem states that
$$m_{\lambda,\bmu}^{(\bX,\bm)} = \left[V(\widehat{\lambda}):V^{\bX^\ast}_{\bm^\ast}(\widehat{\bmu})\right]
\qu{ for all $\la\in \cP_n$.}$$

\medskip

Let 
\[
\bigtriangleup_{(\bX,\bm)}=\prod_{\alpha\in
R_{(\bX,\bm)}}(1-x^{[\alpha]})\text{ where }x^{[\alpha
]}=\left\{
\begin{array}
[c]{l}%
\frac{x_{i}}{x_{j}}\text{ if }\alpha=\varepsilon_{i}-\varepsilon_{j} \text{ where $i<j$}\\
\frac{1}{x_{i}x_{j}}\text{ if }\alpha=\varepsilon_{i}+\varepsilon_{j}\text{ where $i\leq j$}%
\end{array}
\right.
\]

Let $\boldsymbol{\beta} = (\beta^{(1)},\ldots,\beta^{(r)})$ be an $r$-tuple of partitions such that $\beta^{(j)} = (\beta^{(j)}_1,\ldots,\beta^{(j)}_{m_j}) \in \cP_{m_j}$. Then since $\sum m_i = m$,  $\boldsymbol{\beta}$ can be seen as an element of $\Z^m$. We can then define $\sx^{\boldsymbol{\beta}}$ and the map $\sE$ that sends $\sx^{\boldsymbol{\beta}}$ to $\se_{{\boldsymbol{\beta}}}$ as in Section~\ref{Sec_HoweDuality}.

\begin{Lem}
\label{Lemma_EDelta(X,a)} We have 
$\fs_{\bmu}^{(\bX,\bm)}=\sE(\bigtriangleup_{(\bX,\bm)}\sx^{\bmu^{\prime}}).$

\end{Lem}
\begin{proof}
We have 
$$R_{(\bX,\bm)}=\bigsqcup_{j=1}^{r}R^{X_j}_{[M_{j-1}+1,M_{j}]}$$
and 
$$\bigtriangleup_{(\bX,\bm)}\sx^{\bmu^{\prime}}
=\prod_{j=1}^{r}\left(  \prod_{\alpha\in R^{X_j}_{[M_{j-1}+1,M_{j}]}}(1-x^{[\alpha]})\right)\sx^{\bmu'}.$$
Let $\bmu' = ({\mu^{(1)}}',\ldots,{\mu^{(r)}}')$ and ${\mu^{(j)}}' = ({\mu_1^{(j)}}',\ldots,{\mu_{m_j}^{(j)}}')\in \cP_{m_j}$.  For all $j=1,\ldots,r$, we set 
$$\sx^{{\mu^{(j)}}'} =x^{{\mu_1^{(j)}}'}_{M_{j-1}+1}\cdots x^{{\mu_{m_j}^{(j)}}'}_{M_j} $$
Then we obtain
\begin{align*}
\bigtriangleup_{(\bX,\bm)}\sx^{\bmu^{\prime}}
&=\prod_{j=1}^{r}\left(  \prod_{\alpha\in R^{X_j}_{[M_{j-1}+1,M_{j}]}}(1-x^{[\alpha]})\sx^{{\mu^{(j)}}'}\right).
\end{align*}
Observe that the variables appearing in each parenthesised expression are
separated.\ Therefore, we can apply Lemma \ref{Lem_useful} recursively  and
get
\[
\sE(\bigtriangleup_{(\bX,\bm)}\sx^{\bmu^{\prime
}})=\prod_{j=1}^{r}\sE\left(  \prod_{\alpha\in R^{X_j}_{[M_{j-1}+1,M_{j}]}}(1-x^{[\alpha]})\sx^{{\mu^{(j)}}'}\right).\]
But for any $j=1,\ldots,r,$ we have by Propositon \ref{delta_C} and its analogue
for the ordinary Jacobi-Trudi formula%
\[
\fs_{\mu^{(j)}}^{C}=\sE\left( \prod_{\alpha\in R^{X_j}_{[M_{j-1}+1,M_{j}]}}(1-x^{[\alpha]})\sx^{{\mu^{(j)}}'}\right)\qu{when} X_j = C
\]
and
\[
\fs_{\mu^{(j)}}^{A}=\sE\left( \prod_{\alpha\in R^{X_j}_{[M_{j-1}+1,M_{j}]}}(1-x^{[\alpha]})\sx^{{\mu^{(j)}}'}\right)\qu{when} X_j = A.
\]
Finally, we get 
$\fs_{\bmu}^{(\bX,\bm)}=\sE(\bigtriangleup_{(\bX,\bm)}\sx^{\bmu'})$ as desired.
\end{proof}
Define the partition function $\widetilde{\mathcal{P}}_{(\bX,\bm)}$ by
\[
\prod_{\alpha\in R^+\setminus R_{(\bX,\bm)}}\frac{1}{1-x^{[\alpha]}}
=\sum_{\beta\in\mathbb{Z}^{m}}\widetilde{\mathcal{P}}_{(\bX,\bm)}(\beta)\sx^{\beta}.
\]
\begin{Lem}
We have
\[
\widetilde{\mathcal{P}}_{(\bX,\bm)}(\beta)=\mathcal{P}_{(\bX^\ast,\bm^\ast)}(\sI(\beta))\text{ for any }\beta
\in\mathbb{Z}^{m}.
\]
\end{Lem}
\begin{proof}
The involution $\sI$ sends the set $R_{(\bX,\bm)}$ to $R_{(\bX^\ast,\bm^\ast)}$. Therefore we have
\begin{align*}
\sI\left(\prod_{\alpha\in R^+\setminus R_{(\bX,\bm)}}\frac{1}{1-x^{[\alpha]}}\right)
= \prod_{\alpha\in R^+\setminus R_{(\bX^\ast,\bm^\ast)}}\frac{1}{1-x^{\alpha}}
= \sum_{\beta\in \Z^m} \cP_{(\bX^\ast,\bm^\ast)}(\beta)\sx^\beta
\end{align*}
which gives
$$\sum_{\beta\in\mathbb{Z}^{m}}\widetilde{\mathcal{P}}_{(\bX,\bm)}(\beta)\sx^{\beta}=\prod_{\alpha\in R^+\setminus R_{(\bX,\bm)}}\frac{1}{1-x^{[\alpha]}} = \sum_{\beta\in \Z^m} \cP_{(\bX^\ast,\bm^\ast)}(\sI(\beta))\sx^\beta$$
hence the result. 
\end{proof}
\begin{Prop}
For all $\lambda\in \cP_n$, we have 
$$m_{\lambda,\bmu}^{(\bX,\bm)}=\sum_{w\in W_m}\varepsilon(w)\widetilde
{\mathcal{P}}_{(\bX,\bm)}(w(\lambda^{\prime}+\delta_{m,n})-(\bmu^{\prime}+\delta_{m,n}).$$
\end{Prop}

\begin{proof}
By definition of $\Delta^C_m$ (see Section~\ref{Sec_HoweDuality}) we have 
$$ \dfrac{\bigtriangleup_{(\bX,\bm)}}{\Delta^C_m} = \prod_{\alpha\in R^+\setminus R_{(\bX,\bm)}}\frac{1}{1-x^{[\alpha]}}.$$
We obtain
\begin{align*}
s_{\bmu}^{(\bX,\bm)}&=\sE(\bigtriangleup_{(\bX,\bm)}\sx^{\bmu^{\prime}})\tag{Lemma \ref{Lemma_EDelta(X,a)}}\\
&=\sE(\Delta^C_m \cdot \dfrac{1}{\Delta^C_m}\bigtriangleup_{(\bX,\bm)}\sx^{\bmu^{\prime}})\\
&= \sE(\Delta^C_m \sum_{\beta\in\mathbb{Z}^{m}}\widetilde{\mathcal{P}%
}_{(\bX,\bm)}(\beta)\sx^{\beta+\bmu'})\\
&= \sum_{\beta\in\mathbb{Z}^{m}}\widetilde{\mathcal{P}}_{(\bX,\bm)}(\beta)\sE(\Delta_m \sx^{\beta+\bmu'})\\
&= \sum_{\beta\in\mathbb{Z}^{m}}\widetilde{\mathcal{P}}_{(\bX,\bm)}(\beta)\sv_{\beta+\bmu'}.
\end{align*}
We conclude  using Theorem \ref{straightening_C}.
\end{proof}

We are now ready to complete the proof of the main theorem. 

\medskip

\textit{Proof of \Cref{generalised_duality}.}
We have
\begin{align*}
m_{\lambda,\bmu}^{(\bX,\bm)}&=\sum_{w\in W_m}\varepsilon(w)\widetilde
{\mathcal{P}}_{(\bX,\bm)}\left(w(\lambda^{\prime}+\delta_{m,n})-(\bmu^{\prime}+\delta_{m,n})\right)\\
&=\sum_{w\in W_m}\varepsilon(w)\mathcal{P}_{(\bX^\ast,\bm^\ast)}\left(\sI(w(\lambda^{\prime}+\delta_{m,n}))-\sI(\bmu^{\prime}+\delta_{m,n})\right)\\
&=\sum_{w\in W_m}\varepsilon(w)\mathcal{P}_{(\bX^\ast,\bm^\ast)}\left(\sI(w(\lambda^{\prime}+\delta_{m,n}))-\sI(\bmu^{\prime})-n\cdot \om_m - \rho_m\right).
\end{align*}
By setting $w^{\prime}=\sI w\sI$ in the last sum and using Proposition
\ref{Prop_Branching}, this yields%
\begin{align*}
m_{\lambda,\bmu}^{(\bX,\ba)}&=\sum
_{w^{\prime}\in W_m}\varepsilon(w^{\prime})\mathcal{P}_{(\bX^\ast,\bm^\ast)}\left(  w^{\prime}(\sI(\lambda^{\prime})+n\omega
_{m})+\rho_{m})-(\sI(\bmu^{\prime})+n\omega_{m})-\rho_{m}\right)
\\
&=\sum_{w^{\prime}\in W_m}\varepsilon(w^{\prime})\mathcal{P}_{(\bX^\ast,\bm^\ast)}\left(  w^{\prime}(\widehat{\lambda}+\rho
_{m})-(\widehat{\bmu}-\rho_{m})\right)  \\
&=\left[V(\wh{\la}):V^{\bX^\ast}_{\bm^\ast}(\bmu)\right]
\end{align*}
as required. 
\hfill
$\square$

%%%%%%%%%%%%%%%%%%%%%%%%%%%%%%%%%%%%%%%%%%%%%%%%%%
%%%%%%%%%%%%%%%%%%%%%%%%%%%%%%%%%%%%%%%%%%%%%%%%%%
%%%%%%%%%%%%%%%%%%%%%%%%%%%%%%%%%%%%%%%%%%%%%%%%%%
%%%%%%%%%%%%%%%%%%%%%%%%%%%%%%%%%%%%%%%%%%%%%%%%%%
%%%%%%%%%%%%%%%%%%%%%%%%%%%%%%%%%%%%%%%%%%%%%%%%%%
%%%%%%%%%%%%%%%%%%%%%%%%%%%%%%%%%%%%%%%%%%%%%%%%%%

\section{Injectivity of the induction functor}

Theorem \ref{generalised_duality} permits to express tensor multiplicities of
$\mathfrak{sp}_{2n}(\C)$-modules in terms of branching coefficients in
irreducible $\mathfrak{sp}_{2m}$-modules. We started from the tensor product of $\mathfrak{sp}_{2n}(\C)$-modules 
$$V^{X_1}(\mu^{(1)})\otimes V^{X_2}(\mu^{(2)})\otimes \cdots \otimes V^{X_r}(\mu^{(r)})$$
 associated to the sequence $\bmu=(\mu^{(1)},\ldots,\mu^{(r)})$ of partitions such
that $\mu^{(j)}\in \cP_{n\times m_{j}}$ and to the sequence~$\bX=(X_1,\ldots,X_r)$ of symbols in $\{A,C\}$. This determined the dominant weight $\widehat
{\bmu}=\sI(\bmu^{\prime})+n\cdot \omega_{m}$ of the
subalgebra $\mathfrak{g}_{(\bX^\ast,\bm^\ast)}$ of
$\mathfrak{sp}_{2m}(\C)$ where $m=\sum m_i$ and $\bm = (m_1,\ldots,m_r)$.

\medskip

Conversely, we can start from a dominant weight $\wh{\bmu}$ of $\fg_{(\bX^\ast,\bm^\ast)}\subset \frak{sp}_{2m}(\C)$ and realise the associated branching coefficient
as a tensor product multiplicity for $\mathfrak{sp}_{2n}(\C)$-modules. But here,
one has to keep in mind that the datum of $\widehat{\bmu}$ and
$\fg_{(\bX^\ast,\bm^\ast)}$ does not determine the integer $n$.\ One can only say that
$n$ is at least $n_{\wh{\bmu}}$, the greatest
part in the partitions $\wh{\bmu}_{(j)}$ where $\wh{\bmu} = (\wh{\bmu}_{(1)},\ldots,\wh{\bmu}_{(r)})$. This means that for
any integer $n\geq n_{\wh{\bmu}}$, we have an $r$-tuple of partitions
$\bmu[n]=(\mu^{(1)}[n],\ldots,\mu^{(r)}[n])$ such that $\wh{\bmu[n]} = \wh{\bmu}$ (where the map $\wh{\ \cdot\ }$ is defined with respect to the pair $(n,\bm)$).
Observe in particular that for any
$j=1,\ldots,r$, we have
\begin{equation}
\mu^{(j)}[a+n_{\widehat{\bmu}}]=\big(\underset{a\text{ times}}{\underleftrightarrow{m_{j},\ldots,m_{j}}},\mu^{(j)}[n_{\widehat{\bmu}}]\big)
\text{ for any integer }a\geq0 \label{n_trick}%
\end{equation}
or equivalently, $\mu^{(j)}[n+1]$ is obtained by adding a part $m_{j}$ to
$\mu^{(j)}[n]$ for any $n\geq n_{\widehat{\bmu}}$. 
\begin{Exa} 
Consider the following sequence of partition:
$$\wh{\bmu} = \bigg(\ 
\scalebox{.5}{\ydiagram{2,1,1}}\ ,\ 
\scalebox{.5}{\ydiagram{2,2}}\ ,\ 
\scalebox{.5}{\ydiagram{2}}\ \bigg)$$
so that $n_{\wh{\bmu}} = 3$. 
We have seen in Example \ref{widehat} that 
$$\bmu[3] =  \bigg(\ 
\scalebox{.5}{\ydiagram{2,1,1}}\ ,\ 
\scalebox{.5}{\ydiagram{2}}\ ,\ 
\scalebox{.5}{\ydiagram{3,2}}\ \bigg).$$
Now setting $n=5$ and adding $2$ parts of respective size $2,2$ and $3$ to $\bmu[3]$ we get 
 $$\bmu[5] =  \bigg(\ 
\scalebox{.5}{\ydiagram{2,2,2,1,1}}\ ,\ 
\scalebox{.5}{\ydiagram{2,2,2}}\ ,\ 
\scalebox{.5}{\ydiagram{3,3,3,2}}\ \bigg).$$
Then taking the complements in the rectangles of respective size $3\times 5$, $2\times 5$ and $2\times 5$ we obtain
$$ \bigg(\ 
\scalebox{.5}{\ydiagram[*(white)]{2,2,2,1,1}*[*(green!50)]{2,2,2,2,2}}\ ,\ 
\scalebox{.5}{\ydiagram[*(white)]{2,2,2}*[*(green!50)]{2,2,2,2,2}}\ ,\ 
\scalebox{.5}{\ydiagram[*(white)]{3,3,3,2}*[*(green!50)]{3,3,3,3,3}}\ \bigg),$$
and taking the conjugate of each green partition and reversing, we see that we obtain
$\wh{\bmu[5]} = \wh{\bmu}$ 
(here $\wh{\bmu[5]}$ is computed with respect to the pair $(5,(2,2,3))$).  
\end{Exa}
Then, one can
apply Theorem \ref{generalised_duality} and get for all $n\geq n_{\widehat{\bmu}}$ and all
partitions $\lambda\in P_n$ the equality
\[ 
m_{\lambda,\bmu[n]}^{(\bX,\bm)}%
=[V(\widehat{\lambda}):V^{\bX^\ast}_{\bm^\ast}(\widehat{\bmu})].
\]
In the relation above, the partition $\wh{\la}$ is defined from the action of the map $\wh{\ \cdot\ }$ corresponding to the pair $(n,m)$ and 
the $r$-partition $\wh{\bmu}$ is defined similarly with respect to the pair $(n,\bm)$ (in both cases, the same $n$ as in $\bmu[n]$).

\medskip

In the following, we fix two dominant weights $\wh{\bmu},\wh{\bnu}$ of the algebra $\mathfrak{g}_{(\bX^\ast,\bm^\ast)}\subset\mathfrak{sp}_{2m}$. From the definition, we see that $\wh{\bmu}$ and $\wh{\bnu}$ are sequences of partitions in which the $j$-th components lies in $\cP_{m_{r-j+1}}$. We set
$$\wh{\bmu} = (\wh{\mu}_{(1)},\ldots,\wh{\mu}_{(r)})\qu{ and }\wh{\nu} = (\wh{\nu}_{(1)},\ldots,\wh{\nu}_{(r)}).$$
Recall that $V^{\bX^\ast}_{\bm^\ast}(\wh{\bmu})$ and $V^{\bX^\ast}_{\bm^\ast}(\wh{\bnu})$ are the
associated highest weight $\mathfrak{g}_{(\bX^\ast,\bm^\ast)}$-modules. 
We consider the following problem.
\begin{problem}
\label{Problem} Assume that the induction of $V^{\bX^\ast}_{\bm^\ast}(\wh{\bmu})$ and $V^{\bX^\ast}_{\bm^\ast}(\wh{\bnu})$ to $\mathfrak{sp}_{2m}$ are isomorphic. What can we say about $\widehat{\bmu}$ and~$\widehat{\boldsymbol{\nu}}$?
\end{problem}

We are going to prove that in the case where all the components in $\bX^\ast = (X_r,\ldots,X_1)$ are of type
$C$ or when~$\bX^\ast$ is a parabolic Dynkin subdiagram (i.e. $X_{1}$ is the unique
component of type $C$) then the sequences $\widehat
{\bmu}$ and $\widehat{\boldsymbol{\nu}}$ coincide up to
permutations of their parts. Further, in the second case we will show that $\widehat{\mu}_{(r)}=\widehat{\nu}_{(r)}$. 
In particular, this proves the main conjecture of \cite{GL} in type $C$: in the parabolic case, the two previous induced modules 
are isomorphic if and only if the dominant weights $\wh{\bmu},\wh{\bnu}$ coincide up to an automorphism of the underlying parabolic Dynkin diagram. 
In \cite{GL}, this conjecture was proved in any finite types but only when $\wh{\bmu}$ and $\wh{\bnu}$ are far enough from the walls of Weyl chambers.

\medskip

To do this, recall first that the character ring
of type $C_{n}$ can be regarded as the ring $\mathbb{Z}^{W_{n}}[x_{1}^{\pm1},\ldots,x_{n}^{\pm 1}]$ of Laurent polynomials
fixed by permutations of the variables $x_{i}$ and the inversions
$x_{i}\mapsto\frac{1}{x_{i}}$. Also, we have a total order on the
monomials in $\mathbb{Z}[x_{1}^{\pm1},\ldots,x_{n}^{\pm 1}]$ defined by, for all $\beta,\gamma\in \Z^n$,
\[
\sx^{\beta}\leq \sx^{\gamma}\Longleftrightarrow\beta\leq_{\mathrm{lex}}\gamma
\]
where $\leq_{\mathrm{lex}}$ is the lexicographic order on $\mathbb{Z}^{n}$.
This enables us to define, for any $P\in\mathbb{Z}[x_{1}^{\pm1},\ldots,x_{n}^{\pm
1}]$ the monomial $\max(P)$ as the maximal monomial appearing in $P$. Given
$A$ and $B$ in $\mathbb{Z}^{W_{n}}[x_{1}^{\pm1},\ldots,x_{n}^{\pm 1}]$, we
have $\max(AB)=\max(A)\times\max(B)$. This implies the following useful lemma.

\begin{Lem}
\label{Lem_order}Consider $P$ and $Q$ in $\mathbb{Z}[x_{1}^{\pm1},\ldots
,x_{n}^{\pm 1}]$ such that $Q$ divides $P$. Then $\max(Q)\leq\max(P)$.
\end{Lem}

For any $\beta=(\beta_{1},\ldots,\beta_{n})\in\mathbb{Z}^{n}$, set $\left\vert
\beta\right\vert =\beta_{1}+\cdots+\beta_{n}$. Given $P=\sum_{\beta
\in\mathbb{Z}^{n}}c_{\beta}x^{\beta}$ (all but a finite number of coefficients
$c_{\beta}$ are equal to zero) in $\mathbb{Z}^{W_{n}}[x_{1}^{\pm1}%
,\ldots,x_{n}^{\pm 1}]$, define
\[
\mathrm{head}(P)=\sum_{\substack{
\beta\in\mathbb{Z}^{n} \\
\left\vert \beta\right\vert 
=\left\vert \max(P)\right\vert }}
c_{\beta}\sx^{\beta}\text{.}%
\]
We have the following easy lemma.

\begin{Lem}
\label{Lem_head}Let $P_{1},\ldots,P_{r}$ be a sequence of polynomials in
$\mathbb{Z}[x_{1}^{\pm1},\ldots,x_{n}^{\pm 1}]$. Then 
$$\mathrm{head}%
(P_{1}\times\cdots\times P_{r})=\mathrm{head}(P_{1})\times\cdots
\times\mathrm{head}(P_{r}).$$
\end{Lem}

Now we need a result by Rajan \cite{Ra14} on the irreducibility of the
characters for $\mathfrak{sp}_{2n}$. In fact, we do not need Rajan's result in
its full generality and we will only state a weaker version, sufficient for
our purposes. Recall that $\rho_{n}=(n,n-1,\ldots,1)$ and set $\widetilde
{\rho}_{n}=(2n-1,2n-3,\ldots,1)$.

\medskip

For any partition $\delta\in \cP_n$,  let
\begin{itemize}
\item $\fs_{\delta}^{C}$ be the character of the irreducible module $V(\delta)$ of highest weight $\delta$ in $\frak{sp}_{2n}$, 
\item $\fs_{\delta}(X)$ be the usual Schur function associated to $\delta$ (i.e.  the character of the irreducible module $V(\delta)$ of highest weight $\delta$ in $\frak{gl}_{n}$) thus a symmetric polynomial in the set of variables $x_1,\ldots,x_n$,
\item $\fs_{\delta}(X^{\pm 1})$ be the Schur function  $\fs_{\delta}(X^{\pm 1}) = \fs_{\delta}(x_{1},\ldots,x_{n},\frac{1}{x_{n}},\ldots,\frac{1}{x_{1}})$ where $\fs_\delta$ is the Schur function in $2n$ variables.
\end{itemize}

\begin{Th}
\label{Th_Raj}Let $\delta\in \cP_n$ be a dominant weight for $\mathfrak{sp}_{2n}$
regarded as an element of $\mathbb{Z}^{n}$. Assume that the coordinates of
$\delta+\rho_{n}$ are relatively prime (i.e. they have no trivial common
divisor). Then the character $\fs_{\delta}^{C}$ is irreducible in
$\mathbb{C}^{W_{n}}[x_{1}^{\pm1},\ldots,x_{n}^{\pm 1}]$.
\end{Th}

\begin{Rem}
In fact the irreducibility property proved in \cite{Ra14} is more
general.\ Let $d(\delta)$ be the greatest common divisor for the coordinates
of $\delta+\rho_{n}$ on the basis of the fundamental weights $\omega
_{1},\ldots,\omega_{n-1},\omega_{n}$. Equivalently, $d(\delta)$ is the gcd of
the coordinates of $\delta+\rho_{n}$ in the usual basis of $\mathbb{Z}^{n}%
$\footnote{For any $\beta=(\beta_{1},\ldots,\beta_{n})\in\mathbb{Z}$, we
indeed have $\beta=\sum_{i=1}^{n}a_{i}\omega_{i}$ where $a_{i}=\beta_{i}%
-\beta_{i+1}$ for $1\leq i\leq n-1$ and $\beta_{n}=a_{n}$.}. Define also
$\widetilde{d}(\delta)$ as the greatest common divisor for the coordinates of
$\delta+\rho_{n}$ on the weight $\mathbb{Q}$-basis $2\omega_{1},\ldots
,2\omega_{n-1},\omega_{n}$ (with the convention $\widetilde{d}(\delta)=1$ as
soon as we have a non integer coordinate). For any weight $\beta$, recall that
$a_{\beta}=\sum_{w\in W}\varepsilon(w)\sx^{w(\beta+\rho_{n})}$. Then set%
\[
D(\delta)=\gcd(a_{d(\delta)\rho_{n}},a_{\widetilde{d}(\delta)\rho_{n}})
\]
where the greatest common divisor is here considered in $\mathbb{C}^{W_{n}%
}[x_{1}^{\pm1},\ldots,x_{n}^{\pm 1}]$. It is proved in \cite{Ra14} that%
\[
S_{\delta}^{C}:=\frac{a_{\delta+\rho_{n}}}{D(\delta)}%
\]
is irreducible as soon as $\delta+\rho_{n}$ is not a multiple of $\rho_{n}$ or
$\widetilde{\rho}_{n}$. By our assumption in the previous theorem, the
coordinates of $\delta+\rho_{n}$ are relatively prime, therefore $d(\delta)=1$
and $\delta+\rho_{n}$ is not a multiple of $\rho_{n}$. Now, observe that
$a_{\rho_{n}}$ divides $a_{k\rho_{n}}$ for any integer $k\geq1$ because
$\frac{a_{k\rho_{n}}}{a_{\rho_n}}=\fs_{(k-1)\rho_{n}}^{C}$. This implies that
$D(\delta)=\gcd(a_{\rho_{n}},a_{\widetilde{d}(\delta)\rho_{n}})=a_{\rho_{n}}$
and $\fs_{\delta}^{C}=S_{\delta}^{C}$ is irreducible, as claimed.
\end{Rem}

\medskip

Rajan's irreducibility result \cite{Ra14} holds for any finite root
systems. It can be used to prove the second result by Rajan that we shall
need. Again, it holds for any root system but we shall only need it for
the Schur functions. 
\begin{Th}
\label{Th_Rajan2}Let $\lambda^{(1)},\ldots,\lambda^{(r)}$ and $\mu
^{(1)},\ldots,\mu^{(r)}$ be two sequences of partitions in $\cP_n$ such
that%
\[
\fs_{\lambda^{(1)}}(X)\times\cdots\times \fs_{\lambda^{(r)}}(X)=\fs_{\mu^{(1)}%
}(X)\times\cdots\times \fs_{\mu^{(r)}}(X).
\]
Then we have the multiset equality $\{\lambda^{(1)},\ldots,\lambda
^{(r)}\}=\{\mu^{(1)},\ldots,\mu^{(r)}\}$.
\end{Th}

One can observe that this result on Schur functions easily implies its analogue
on Weyl characters thanks to Lemma \ref{Lem_head} and the simple observation%
\[
\mathrm{head}(\fs_{\nu}^{C})=\fs_{\nu}(X)\text{ for all }\nu\in\mathcal{P}%
_{n}.
\]

\begin{Cor}
\label{Cor_rajtypeC}Consider $\lambda^{(1)},\ldots,\lambda^{(r)}$ and
$\mu^{(1)},\ldots,\mu^{(r)}$ two sequences of partitions in $\mathcal{P}_n$ such
that%
\[
\fs_{\lambda^{(1)}}^{C}\times\cdots\times \fs_{\lambda^{(r)}}^{C}%
=\fs_{\mu^{(1)}}^{C}\times\cdots\times \fs_{\mu^{(r)}}^{C}.
\]
Then we have the multiset equality $\{\lambda^{(1)},\ldots,\lambda
^{(r)}\}=\{\mu^{(1)},\ldots,\mu^{(r)}\}$.
\end{Cor}

Now recall the branching formula for the decomposition of a Schur function  $\fs_\nu(X^{\pm 1})$
 on the basis of the irreducible characters for $\mathfrak{sp}_{2n}$. 

\begin{Th}[{\cite[Appendix A]{FH91}}]
Let $\nu\in\mathcal{P}_{n}$. We have
$$\fs_{\nu}(X^{\pm1})=\sum_{\lambda
\in\mathcal{P}_{n},\delta\in\mathcal{P}_{n}^{(1,1)}}c_{\lambda,\delta}^{\nu
}\fs_{\lambda}^{C}.$$
where $\mathcal{P}_{n}^{(1,1)}$ is the subset of $\mathcal{P}_{n}$ of
partitions which can be tiled in vertical dominoes and $c_{\lambda,\delta
}^{\nu}$ is the usual Littlewood-Richardson coefficient.
\end{Th}

\begin{Rem}
\label{Rem_divide_equal}Observe that $\fs_{\nu}(X^{\pm
1})\neq \fs_{\nu}^{C}$ as soon as $\nu$ has at least two rows since
$c_{\lambda,(1,1)}^{\nu}>0$ for any $\lambda$ obtained from $\nu$ by removing
one box in two different rows. The previous decomposition can then be written%
\[
\fs_{\nu}(X^{\pm1})=\fs_{\nu}^{C%
}+\sum_{\lambda\vartriangleleft\nu}a_{\nu,\lambda}\fs_{\lambda}^{C}%
\]
where the coefficients $a_{\nu,\lambda}$ belong to $\mathbb{N}$ and
$\trianglelefteq$ is the dominant order on partitions. Assume that $\fs_{\nu
}^{C}$ divides $\fs_{\nu}(X^{\pm 1})$ in
$\mathbb{Z}[X^{\pm 1}]$. Then $\fs_{\nu}^{C}$ divides
$\sum_{\lambda\vartriangleleft\nu}a_{\nu,\lambda}\fs_{\lambda}^{C}$. But
$\max(\fs_{\nu}^{C})=\sx^{\nu}$ and $\max(\sum_{\lambda\vartriangleleft\nu
}a_{\nu,\lambda}\fs_{\lambda}^{C})=\sx^{\nu^{\flat}}$ where $\nu^{\flat}$ is
the partition obtained by decreasing by $1$ the two lowest nonzero parts of $\nu
$. Since $\sx^{\nu}>\sx^{\nu^{\flat}}$, this contradicts Lemma \ref{Lem_order}.
This shows that $\fs_{\nu}^{C}$ divides $\fs_{\nu}(X^{\pm 1})$ in $\mathbb{Z}[X^{\pm 1}]$ if and only
if $\fs_{\nu}^{C}=\fs_{\nu}(X^{\pm1})$.
\end{Rem}

Now let us come back to\ Problem \ref{Problem} and assume that the induction of $V^{\bX^\ast}_{\bm^\ast}(\wh{\bmu})$ and $V^{\bX^\ast}_{\bm^\ast}(\wh{\bnu})$ to $\mathfrak{sp}_{2m}$ are isomorphic. This can be rewritten, thanks to our main theorem in the last section as:
$$\fs^{\bX,\bm}_{\bmu[n]} = \fs_{\bnu[n]}^{(\bX,\bm)}\qu{ for all $n\geq\max(n_{\widehat{\bmu}},n_{\widehat{\boldsymbol{\nu}}}$).}$$
Assume first that all the components of $\bX$ are of type $C$. We get%
\begin{equation}
\fs_{\mu^{(1)}[n]}^{C}\times\cdots\times \fs_{\mu^{(r)}[n]}^{C}%
=\fs_{\nu^{(1)}[n]}^{C}\times\cdots\times \fs_{\nu^{(r)}[n]}^{C}
\label{equal_product}%
\end{equation}
where for any $j=1,\ldots,r,$ the partitions $\mu^{(j)}[n]$ and $\nu^{(j)}[n]$
have at most $m_{j}$-columns. Here we can apply Corollary \ref{Cor_rajtypeC}
and deduce the multiset equality 
$\{\mu^{(1)}[n],\ldots,\mu^{(r)}[n]\}=\{\nu^{(1)}[n],\ldots,\nu^{(r)}[n]\}$.

\medskip

Alternatively, we can choose $n$ so that
each partition $\mu^{(j)}[n]$ and $\nu^{(j)}[n]$ starts with two parts equal
to $m_{j}$. Then, the two first components in $\mu^{(j)}[n]+\rho_{n}$ or
$\nu^{(j)}[n]+\rho_{n}$ are equal to $m_{j}+n$ and $m_{j}+n-1$,
respectively.\ In particular, this implies that the coordinates of $\mu
^{(j)}[n]+\rho_{n}$ and $\nu^{(j)}[n]+\rho_{n}$ are relatively prime. We thus
get by Theorem \ref{Th_Raj} that each character $\fs_{\mu^{(j)}[n]}^{C}$ or
$\fs_{\nu^{(j)}[n]}^{C}$ are irreducible. Thus we recover the multiset
equality $\{\mu^{(1)}[n],\ldots,\mu^{(r)}[n]\}=\{\nu^{(1)}[n],\ldots,\nu
^{(r)}[n]\}$.\ It implies that the $r$-partitions $\widehat{\bmu}$ and
$\widehat{\boldsymbol{\nu}}$ coincide up to permutation of their partitions.

\medskip

Next, assume $\bX=(C,A,\ldots,A)$. This time we get
\begin{equation}
\fs_{\mu^{(1)}[n]}^{C}\times \fs_{\mu^{(2)}[n]}(X^{\pm1})\times\cdots\times
\fs_{\mu^{(r)}[n]}(X^{\pm1})=\fs_{\nu^{(1)}[n]}^{C}\times \fs_{\nu^{(2)}%
[n]}(X^{\pm1})\times\cdots\times \fs_{\nu^{(r)}[n]}(X^{\pm1}).
\label{Equal_Levi}%
\end{equation}
Write $N=\left\vert \mu^{(1)}[n]\right\vert +\cdots+\left\vert \mu
^{(r)}[n]\right\vert $ the sum of all parts in the partitions $\mu
^{(j)}[n],j=1,\ldots,r$ (that is the rank of the $r$-partition
$\bmu[n]$).\ For any partition $\kappa\in\mathcal{P}_{n}$, we have
$$\max(\fs_{\kappa}^{C})=\max(\fs_{\kappa}(X))=\sx^{\kappa}\qu{ and }\mathrm{head}%
(\fs_{\kappa}^{C})=\fs_{\kappa}(X).$$ 
Then Lemma \ref{Lem_head} implies that
for any partition $\nu$ of rank $N$, the coordinates of 
$$\fs_{\mu^{(1)}(n)}^{C}\times \fs_{\mu^{(2)}(n)}(X^{\pm1})\times\cdots\times \fs_{\mu
^{(r)}(n)}(X^{\pm1})$$
on the irreducible character $\fs_{\nu}^{C}$ coincides
with the coordinates of 
$$\fs_{\mu^{(1)}(n)}(X)\times \fs_{\mu^{(2)}(n)}(X)\times\cdots\times \fs_{\mu^{(r)}(n)}(X)$$
on the Schur function $\fs_{\nu}(X)$.
Equation (\ref{Equal_Levi}) then gives 
\[
\fs_{\mu^{(1)}(n)}(X)\times \fs_{\mu^{(2)}(n)}(X)\times\cdots\times \fs_{\mu
^{(r)}(n)}(X)=\fs_{\nu^{(1)}(n)}(X)\times \fs_{\nu^{(2)}(n)}(X)\times\cdots\times
\fs_{\nu^{(r)}(n)}(X)
\]
and we obtain $\{\mu^{(1)}[n],\ldots,\mu^{(r)}[n]\}=\{\nu^{(1)}[n],\ldots,\nu^{(r)}[n]\}$
  by Theorem \ref{Th_Rajan2}.

\medskip

Having in hand the multiset equality $\{\mu^{(1)}[n],\ldots,\mu^{(r)}%
[n]\}=\{\nu^{(1)}[n],\ldots,\nu^{(r)}[n]\}$ and equation (\ref{Equal_Levi}%
), we deduce that

\begin{enumerate}
\item Either $\mu^{(1)}[n]=\nu^{(1)}[n]$ and $\{\mu^{(2)}[n],\ldots,\mu
^{(r)}[n]\}=\{\nu^{(2)}[n],\ldots,\nu^{(r)}[n]\}$,

\item or $\mu^{(1)}[n]=\nu^{(p)}[n],$ $\mu^{(q)}[n]=\nu^{(1)}[n]$ with $p>1$
and $q>1$ and 
$$\{\mu^{(2)}[n],\ldots,\mu^{(r)}[n]\}\backslash\{\mu
^{(q)}[n]\}=\{\nu^{(2)}[n],\ldots,\nu^{(r)}[n]\}\setminus\{\nu^{(p)}[n]\}.$$
\end{enumerate}

In Case (1), since 
$$\wh{\bmu[n]} = \wh{\bmu}\qu{and}\begin{cases}
\wh{\bmu[n]} = (\wh{\mu^{(r)}[n]},\ldots,\wh{\mu^{(1)}[n]})\\
 \wh{\bmu} = (\wh{\mu}_{(1)},\ldots,\wh{\mu}_{(r)})
\end{cases}$$ we get that $\wh{\mu}_{(r)} = \wh{\nu}_{(r)}$ and  $\{\wh{\mu}^{(1)},\ldots,\wh{\mu}^{(r-1)}\} = \{\wh{\nu}^{(1)},\ldots,\wh{\nu}^{(r-1)}\}$.

\medskip

In Case (2), by simplifying the
identical factors in (\ref{Equal_Levi}), we get%
\[
\fs_{\mu^{(1)}[n]}^{C}\times \fs_{\mu^{(q)}[n]}(X^{\pm1})=\fs_{\mu^{(q)}%
[n]}^{C}\times \fs_{\mu^{(1)}[n]}(X^{\pm1}).
\]
As before, we can choose the integer $n$ so that $\fs_{\mu^{(1)}[n]}^{C%
}$ and $\fs_{\nu^{(1)}[n]}^{C}=\fs_{\mu^{(q)}[n]}^{C}$ are irreducible
polynomials, and the partitions $\mu^{(1)}[n]$ and $\nu^{(1)}[n]$ have their
two first parts equal to $m_{1}$. This implies that $\fs_{\mu^{(1)}[n]}^{C}$
divides $\fs_{\mu^{(1)}[n]}(X^{\pm1})$. But, by using Remak
\ref{Rem_divide_equal}, this is only possible when $\fs_{\mu^{(1)}[n]}^{C%
}=\fs_{\mu^{(1)}[n]}(X^{\pm1})$, that is when $\mu^{(1)}[n]$ is a row. We thus
get a contradiction since $\mu^{(1)}[n]$ has at least two parts equal to
$m_{1}$.

\medskip

Finally, we have proved the following theorem.

\begin{Th}
Assume that the induction of $V^{\bX^\ast}_{\bm^\ast}(\wh{\bmu})$ and $V^{\bX^\ast}_{\bm^\ast}(\wh{\bnu})$ to $\mathfrak{sp}_{2m}$ are isomorphic. Then the following holds.
\begin{enumerate}
\item The $r$-partitions $\widehat{\bmu}$ and $\widehat
{\boldsymbol{\nu}}$ coincide up to permutation of their partitions when all
the components in $\bX$ are of type $C$
\item We have $\widehat{\mu}^{(r)}=\widehat{\nu}^{(r)}$ and $\widehat
{\bmu}$ and $\widehat{\boldsymbol{\nu}}$ coincide up to
permutation of their other partitions when $\bX$ is a parabolic Dynkin
subdiagram (i.e. $X_{1}$ is the unique component equal to $C$)
\end{enumerate}
\end{Th}

%%%%%%%%%%%%%%%%%%%%%%%%%%%%%%%%%%%%%%%%%%%%%%%%%%
%%%%%%%%%%%%%%%%%%%%%%%%%%%%%%%%%%%%%%%%%%%%%%%%%%
%%%%%%%%%%%%%%%%%%%%%%%%%%%%%%%%%%%%%%%%%%%%%%%%%%
%%%%%%%%%%%%%%%%%%%%%%%%%%%%%%%%%%%%%%%%%%%%%%%%%%
%%%%%%%%%%%%%%%%%%%%%%%%%%%%%%%%%%%%%%%%%%%%%%%%%%
%%%%%%%%%%%%%%%%%%%%%%%%%%%%%%%%%%%%%%%%%%%%%%%%%%

\end{document}